\documentclass[a4paper,reqno]{amsart}
\usepackage{amscd,amsfonts,amssymb}
\usepackage{tikz}
\usepackage{pgfplots}
\usetikzlibrary{automata,positioning,calc,trees}
\usetikzlibrary{intersections,pgfplots.fillbetween}
\usepackage{graphicx,tikz}
\usepackage{pgf,tikz,pgfplots}
\usepackage{mathrsfs}
\usetikzlibrary{arrows}
\usepackage{enumerate}
\usepackage[shortlabels]{enumitem}
\usepackage{mathrsfs}
\usepackage{amssymb,amsmath,amsthm,color}
\usepackage{caption,subcaption}
\usepackage{hyperref}
\usepackage{cleveref}
\usepackage[alphabetic,bibtex-style]{amsrefs}
\usepackage{url}
\usepackage{setspace}
\usepackage{float}
\usepackage{makecell}
\tikzstyle{startstop} = [rectangle, rounded corners, minimum width=3cm, minimum height=1cm,text centered, draw=black]
\tikzstyle{process} = [rectangle, rounded corners, minimum width=3cm, minimum height=1cm, text centered, draw=black]
\tikzstyle{arrow} = [thick,->,>=stealth]

\BibSpec{article}{%
	+{}  {\PrintAuthors}                {author}
	+{,} { \textit}                     {title}
	+{.} { }                            {part}
	+{:} { \textit}                     {subtitle}
	+{,} { \PrintContributions}         {contribution}
	+{.} { \PrintPartials}              {partial}
	+{,} { }                            {journal}
	+{}  { \textbf}                     {volume}
	+{}  { \PrintDatePV}                {date}
	+{,} { \issuetext}                  {number}
	+{,} { \eprintpages}                {pages}
	+{,} { }                            {status}
	+{,} { \url}                        {url}    
	+{,} { \PrintDOI}                   {doi}
	+{,} { available at \eprint}        {eprint}
	+{}  { \parenthesize}               {language}
	+{}  { \PrintTranslation}           {translation}
	+{;} { \PrintReprint}               {reprint}
	+{.} { }                            {note}
	+{.} {}                             {transition}
}
\newtheorem{theorem}{Theorem}[section]

\newtheorem{lemma}[theorem]{Lemma}
\newtheorem{proposition}[theorem]{Proposition}
\theoremstyle{definition}

\theoremstyle{remark}
\newtheorem{claim}[theorem]{Claim}

\numberwithin{equation}{section}

\newcommand{\R}{{\mathbb {R}}}
\newcommand{\Q}{{\mathbb Q}}
\newcommand{\N}{{\mathbb N}}
\newcommand{\Z}{{\mathbb Z}}

\newcommand{\Sp}{{\mathbb S}}
\newcommand{\cosec}{\mathrm{cosec}}

\newcommand{\vertiii}[1]{{\left\vert\kern-0.25ex\left\vert\kern-0.25ex\left\vert #1 
		\right\vert\kern-0.25ex\right\vert\kern-0.25ex\right\vert}}
\makeatletter
\newcommand{\biggg}{\bBigg@\thr@@}
\newcommand{\Biggg}{\bBigg@{3.5}}

\makeatother
\begin{document}
	\title{On a variant of bilinear spherical maximal function}
	
	\author[Bhojak]{Ankit Bhojak}
	\address{Ankit Bhojak\\
		Department of Mathematics\\
		Indian Institute of Science Education and Research Bhopal\\
		Bhopal-462066, India.}
	\email{ankitb@iiserb.ac.in}
	
	\author[Choudhary]{Surjeet Singh Choudhary}
	\address{Surjeet Singh Choudhary\\
		Department of Mathematical Sciences\\
		Indian Institute of Science Education and Research Mohali\\
		Mohali-140306, India.}
	\email{surjeet19@iisermohali.ac.in}
	
	\author[Shrivastava]{Saurabh Shrivastava}
	\address{Saurabh Shrivastava\\
		Department of Mathematics\\
		Indian Institute of Science Education and Research Bhopal\\
		Bhopal-462066, India.}
	\email{saurabhk@iiserb.ac.in}
	
	\thanks{}
	\begin{abstract} We obtain $L^p-$estimates for the full and lacunary maximal functions associated to the bilinear spherical averages given by
			\[\mathfrak{A}_t(f_1,f_2)(x,y)=\int_{\mathbb S^{2d-1}}f_1(x+tz_1,y)f_2(x,y+tz_2)\;d\sigma(z_1,z_2),\;t>0,\]
			for all dimensions $d\geq1$. We show that the estimates for such operators in dimensions $d\geq2$ essentially rely on the method of slicing. The bounds for the lacunary maximal function in dimension one are more delicate and require a trilinear smoothing inequality, which is based on an appropriate sublevel set estimate in this context.
	\end{abstract}
	\maketitle
	
	\section{Introduction and main results}
	The goal of this paper is to establish $L^p$-estimates for a variant of bilinear spherical maximal function. This study is motivated by recent papers on bilinear spherical maximal functions and bilinear singular integral operators along polynomials. First, we briefly discuss some of the recent works which are relevant to the current paper, then we introduce the bilinear spherical maximal functions and provide statements of the main results obtained in the paper. 
	
	For any $d\geq2$, $\sigma=\sigma_{d-1}$ denote the normalized surface measure on the unit sphere $\mathbb S^{d-1}\subset \R^d$. For $f_1,f_2\in\mathscr{S}(\R^{d})$ (the space of Schwartz functions in $\R^d$), and $d\geq 1$, the full bilinear spherical maximal function is defined by 
	$$\mathcal{M}(f,g)(x):=\sup_{t>0}\left|\int_{\mathbb S^{2d-1}}f(x-ty_1)g(x-ty_2)\;d\sigma(y_1,y_2)\right|.$$
	This operator was first studied by Geba et al. \cite{GGIPS}. Jeong-Lee~\cite{JeongLee} introduced a method of ``slicing" to obtain sharp $L^p$-estimates for the full maximal function $\mathcal{M}(f,g)$. The slicing argument allows us to dominate the full bilinear spherical maximal function by a product of the classical Hardy-Littlewood maximal function and full (linear) spherical maximal function defined as
	\[\mathcal Af(x)=\sup_{t>0}\left|\int_{\mathbb S^{d-1}}f(x-ty)\;d\sigma(y)\right|.\]
	The operator $A$ is bounded in $L^p(\R^d)$- if, and only if $p>\frac{d}{d-1}$, see Stein \cite{MaximalFunctionsISphericalMeans} for dimension $d\geq3$ and Bourgain \cite{BourgainCircular} for $d=2$. For $d\geq 2$, Jeong-Lee~\cite{JeongLee} used the $L^p$-estimates for linear maximal operators along with the slicing argument to conclude that the operator $\mathcal{M}$ is bounded from $L^{p_1}(\R^d)\times L^{p_2}(\R^d)$ to $L^{p}(\R^d)$ if, and only if $p>\frac{d}{2d-1}$ and $\frac{1}{p_1}+\frac{1}{p_2}=\frac{1}{p}$, except at the points $(1,\infty,1)$ and $(\infty,1,1)$.  The case of dimension $d=1$ was later addressed by Christ-Zhou~ \cite{ChristZhou}, also see~Dosidis-Ramos \cite{DosidisRamos}, Shrivastava-Shuin \cite{ShrivastavaShuin}. They proved that $\mathcal{M}$ is bounded from $L^{p_1}(\R)\times L^{p_2}(\R)\to L^p(\R)$ for $p_1,p_2>2$ and $\frac{1}{p_1}+\frac{1}{p_2}=\frac{1}{p}$. 
	
	Another important maximal function associated with bilinear spherical averages is the lacunary maximal function. It is defined by taking supremum over lacunary sequences as follows  
	\[\mathcal A_{lac}f(x)=\sup_{k\in \Z}\left|\int_{\mathbb S^{d-1}}f(x-2^ky)\;d\sigma(y)\right|.\]
	The $L^p$-estimates for the lacunary maximal operator $\mathcal{A}_{lac}$ holds for a larger range $p>1$.  We refer the reader to  Calder\'{o}n \cite{CalderonLacunarySphericalMeans}, Coifman-Weiss \cite{CoifmanWeissBookReview}, Duoandikoetxea-Rubio de Francia \cite{DR} for $L^p$-boundedness of the lacunary maximal operator $\mathcal{A}_{lac}$ for $p>1$. Lacey~\cite{SparseBoundsForSphericalMaximalFunction} strengthened $L^p$-boundedness properties of spherical maximal operators further by establishing sparse domination of $\mathcal {A}$ and $\mathcal A_{lac}$. 
	However, the weak-type boundedness of  $\mathcal{A}_{lac}$ at the end-point $p=1$ remains a long-standing open problem. We refer the reader to~Christ~\cite{Christ1985}, Seeger-Tao-Wright~\cite{EndpointMappingPropertiesOfSphericalMaximalOperators} and Cladek-Krause~\cite{CladekKrause} for some interesting results for the lacunary operator $\mathcal{A}_{lac}$ near $L^1$. 
	
	The lacunary bilinear spherical maximal function is defined by 
	\[\mathcal{M}_{lac}(f,g)(x):=\sup_{k\in \Z}\left|\int_{\mathbb S^{2d-1}}f(x-2^ky_1)g(x-2^ky_2)\;d\sigma(y_1,y_2)\right|.\]
	The $L^p$-estimates for the lacunary operator $\mathcal{M}_{lac}$ hold for a larger range of $p'$s as compared with the case of full maximal function. Borges-Foster \cite{Borges} proved the $L^p$-estimates for $\mathcal{M}_{lac}$ in dimensions $d\geq2$ and Christ-Zhou \cite{ChristZhou} in dimension $d=1$. The case of dimension $d\geq 2$ exploits the slicing method, which is not applicable in dimension $d=1$. Further, the decay estimates on the Fourier transform $\widehat{d\sigma}$ do not help in extending $L^p$-estimates for the full range of $p'$s in dimension $d=1$. Therefore, the case of $d=1$ is significantly challenging to deal with. This motivates to studying trilinear smoothing inequalities studied by Christ in the articles \cite{Christ2, Christ3, Christ1}. Christ-Zhou~\cite{ChristZhou} proved trilinear smoothing estimates for bilinear averaging operators along certain curves, of which unit circle in the plane is a prime example. These trilinear smoothing estimates play the key role in obtaining $L^p$-estimates for the operator $\mathcal{M}_{lac}$. We also refer the reader to ~Bhojak-Choudhary-Shrivastava-Shuin~\cite{BCSS1}, Bak-Shim~\cite{BakShim} and Lee-Shuin~\cite{LeeShuin} for some recent developments related to bilinear spherical maximal functions. 
	
	The trilinear smoothing inequalities are crucial in investigating $L^p$-boundedness of a large class of maximal bilinear singular integral operators along curves. Recently,  Christ-Durcik-Roos~\cite{CDR} studied the bilinear Hilbert transform along parabola 
	\[H(f_1,f_2)(x,y):=~\text{p.v.}~\int_{\R}f_1(x+t, y)f_2(x, y+t^2)\frac{dt}t,~~ f_1, f_2\in \mathscr S(\R^2).\]
	The operator $H$ is a variant of the triangular Hilbert transform
	\[(f_1,f_2)\mapsto \text{p.v.}~\int_{\R}f_1(x+t,y)f_2(x,y+t)\frac{dt}t.\]
	We also refer to the works of Li \cite{Li2013}, Lie \cite{Lie2015}, Li-Xiao \cite{LiXiao}, and Lie \cite{Lie2018} for the boundedness of the standard bilinear Hilbert transform along curves. In \cite{CDR}, they proved trilinear smoothing inequalities in this context and consequently obtained $L^p$-boundedness of the operator $H$. Further, a very interesting application of the trilinear smoothing inequalities in the context of quantitative bounds on Roth-type theorem with corners determined by a parabola in subsets of the unit square $[0,1]^2$ was also obtained. We also refer to the study guide by Hsu-Lin-Stokolosa~\cite{studyguidetrilinearsmoothing} for a nice exposition of this results. We also refer to Gaitan-Lie \cite{GaitanLie2024} for a different approach to obtain $L^p-$estimates for $H$. Later, Chen-Guo~\cite{CG} extended the notion of the bilinear Hilbert transform along parabola to more general polynomials. In particular, they studied the operator 
	\[H_{P_1,P_2}(f_1,f_2)(x,y):=\text{p.v.}\int_{\R}f_1(x+P_1(t), y)f_2(x, y+P_2(t))\frac{dt}t,~~ f_1, f_2\in \mathscr S(\R^2),\]
	where $P_1$ and $P_2$ are linearly independent polynomials with zero constant term.  They proved a quantitative Roth-type theorem with corners determined by polynomial curves. More precisely, they obtained quantitative lower bounds on the gap parameter $t$ so that configurations of the type $(x,y), (x+P_1(t), y), (x, y+P_2(t))$ lie in positive measure subsets in the unit square $[0,1]^2.$ The trilinear smoothing inequalities are the main tools in proving the Roth-type theorems, see~\cite{CDR, CG}. We also refer to \cite{HL2024} for a short proof of triangular Hilbert transform along curves in the Banach range.
	
	\subsection{Main results} Motivated by these fascinating developments in the theory of bilinear singular integral operators along curves and their applications to Roth-type theorems, we investigate the $L^p$-estimates of a variant of bilinear spherical maximal function. These operators were defined by Christ-Zhou~\cite{ChristZhou}. For $f_1,f_2\in\mathscr{S}(\R^{2d}), d\geq 1$, and $t>0$, the bilinear spherical average of $f_1$ and $f_2$ is defined by
	\[\mathfrak{A}_t(f_1,f_2)(x,y)=\int_{\mathbb S^{2d-1}}f_1(x+tz_1,y)f_2(x,y+tz_2)\;d\sigma(z_1,z_2).\]
	
	The full bilinear maximal function associated with the bilinear averages $\mathfrak{A}_t(f_1,f_2)$ is defined by
	\[\mathfrak{M}(f_1,f_2)(x,y):=\sup_{t>0}\left|\mathfrak{A}_t(f_1,f_2)(x,y)\right|.\]
	Further, the lacunary bilinear maximal function associated with the bilinear averages $\mathfrak{A}_t(f_1,f_2)$ is given by
	\[\mathfrak{M}_{lac}(f_1,f_2)(x,y):=\sup_{j\in\Z}\left|\mathfrak{A}_{2^j}(f_1,f_2)(x,y)\right|.\]
	The following theorems are our main results on $L^p$-boundedness of the operators $\mathfrak{M}$ and $\mathfrak{M}_{lac}.$
	\begin{theorem}\label{full}
		Let $1\leq p_1,p_2\leq \infty$ and $\frac{1}{p_1}+\frac{1}{p_2}=\frac{1}{p}$. Then we have
		\[\|\mathfrak{M}(f_1,f_2)\|_{L^p(\R^{2d})}\lesssim \|f_1\|_{L^{p_1}(\R^{2d})}\|f_2\|_{L^{p_2}(\R^{2d})},\]
		for all $f_1\in L^{p_1}(\R^{2d}),\;f_2\in L^{p_2}(\R^{2d}),$ and
		\begin{enumerate}
			\item $d\geq2$ and $\frac{d}{2d-1}<p\leq\infty$, or
			\item $d=1$ and $2<p_1,p_2\leq\infty$.
		\end{enumerate}
		Moreover, we have the following restricted weak-type inequality
		\[\|\mathfrak{M}(f_1,f_2)\|_{L^{p,\infty}(\R^{2d})}\lesssim \|f_1\|_{L^{p_1,1}(\R^{2d})}\|f_2\|_{L^{p_2,1}(\R^{2d})},\]
		for all $f_1\in L^{p_1}(\R^{2d}),\;f_2\in L^{p_2}(\R^{2d}),$ and
		\begin{enumerate}
			\item $d\geq3$ and $p=\frac{d}{2d-1}$, or
			\item $d=2$, $p=\frac{2}{3}$ and $1< p_1,p_2<2$, or
			\item $d=1$, $1\leq p\leq\infty$ and $p_1=2$ or $p_2=2$.
		\end{enumerate}
	\end{theorem}
	\begin{theorem}\label{lacunary}
		Let $d\geq1$ and $\mathsf{p}_d=\frac{2d+1}{2d}$, $1<p_1,p_2<\infty$ and $\frac{1}{p_1}+\frac{1}{p_2}=\frac{1}{p}$. Then we have
		\[\|\mathfrak{M}_{lac}(f_1,f_2)\|_{L^p(\R^{2d})}\lesssim \|f_1\|_{L^{p_1}(\R^{2d})}\|f_2\|_{L^{p_2}(\R^{2d})},\]
		for all $f_1\in L^{p_1}(\R^{2d})$ and $f_2\in L^{p_2}(\R^{2d})$ whenever
		\begin{enumerate}
			\item $d\geq2$ and $\left(\frac{1}{p_1},\frac{1}{p_2},\frac{1}{p}\right)$ belongs to the interior of the convex hull generated by the points $(0,0,0),\;(1,0,1)$, $\left(1,\frac{d-1}{d},\frac{d}{2d-1}\right),\;\left(\frac{1}{\mathsf{p}_d},\;\frac{1}{\mathsf{p}_d},\;\frac{2}{\mathsf{p}_d}\right)$, $\left(\frac{d-1}{d},1,\frac{d}{2d-1}\right),$ and $(0,1,1)$.
			\item $d=1$ and $\left(\frac{1}{p_1},\frac{1}{p_2},\frac{1}{p}\right)$ belongs to the interior of the convex hull generated by the points $(0,0,0),\;\left(\frac{1}{2},0,\frac{1}{2}\right)$, $\left(\frac{2}{3},\;\frac{2}{3},\;\frac{4}{3}\right),$ and $\left(0,\frac{1}{2},\frac{1}{2}\right)$.
		\end{enumerate}
	\end{theorem}
	\begin{figure}[H]	
		\centering
		\begin{tikzpicture}[scale=3]
			\fill[gray!40] (0,0)--(0.5,0)--(0.5,0.5)--(0,0.5)--cycle;
			\fill[blue!10] (0.5,0)--(0.66,0.66)--(0,0.5)--(0.5,0.5)--cycle;
			\draw[thin][->]  (0,0)node[left]{$(0,0)$} --(1.15,0) node[right]{$\frac{1}{p_1}$};
			\draw[thin][->]  (0,0) --(0,1.15) node[above]{$\frac{1}{p_2}$};
			\draw[dotted] (0,1)--(1,1);
			\draw[dotted] (1,0)--(1,1);
			\draw[dotted] (0.5,0)--(0.5,0.5)node[right]{$(\frac{1}{2},\frac{1}{2})$};
			\draw[dotted] (0.5,0.5)--(0,0.5);
			\draw[dotted] (0.5,0)--(0.66,0.66)node[right]{$(\frac{2}{3},\frac{2}{3})$};
			\draw[dotted] (0,0.5)--(0.66,0.66);
			\filldraw [black] (0.5,0)node[below]{$(\frac{1}{2},0)$};
			\filldraw [black] (0,0.5)node[left]{$(0,\frac{1}{2})$};
			\node at (0.5,-0.5) {\textit{Case I: $d=1$.}};
		\end{tikzpicture}
		\qquad
		\begin{tikzpicture}[scale=3]
			\fill[gray!40] (0,0)--(1,0)--(1,0.5)--(0.5,1)--(0,1)--cycle;
			\fill[blue!10] (1,0.5)--(0.85,0.85)--(0.5,1)--cycle;
			\draw[thin][->]  (0,0)node[left]{$(0,0)$} --(1.15,0) node[right]{$\frac{1}{p_1}$};
			\draw[thin][->]  (0,0) --(0,1.15) node[above]{$\frac{1}{p_2}$};
			\draw[dotted] (0,1)--(1,1);
			\draw[dotted] (1,0)--(1,1);
			\draw[dotted] (1,0.5)--(0.5,1);
			\draw[dotted] (1,0.5)--(0.85,0.85)node[right]{$\left(\frac{1}{\mathsf{p}_d},\;\frac{1}{\mathsf{p}_d}\right)$};
			\draw[dotted] (0.5,1)--(0.85,0.85);
			\filldraw [black] (1,0)node[below]{(1,0)};
			\filldraw [black] (1,0.5)node[right]{$(1,\frac{d-1}{d})$};
			\filldraw [black] (0.5,1)node[above]{$(\frac{d-1}{d},1)$};
			\filldraw [black] (0,1)node[left]{(0,1)};
			\node at (0.5,-0.5) {\textit{Case II: $d\geq2$.}};
		\end{tikzpicture}
		\caption{The grey region denotes the boundedness of $\mathfrak{M}$ and the grey and blue regions denotes the boundedness of $\mathfrak{M}_{lac}$.}
		\label{Fig:slicedbilineardiagonal}	
	\end{figure}
	The proof of \Cref{full} is based on the slicing arguments from~\cite{JeongLee, BCSS1}. The proof of \Cref{lacunary} is significantly more technical and requires many ingredients. It is carried out in multiple steps. We build the proof exploiting the ideas developed in ~\cite{Bernicot,Borges,CG, CDR, ChristZhou, IPS}. In the case of dimension $d=1$, the following trilinear smoothing inequality for the bilinear operator plays a key role in the proof of \Cref{lacunary}. 
	\begin{theorem}\label{Trilinearsmoothing}
		Consider the operator defined by
		\[
		T\left(f_1, f_2\right)(x, y)=\int_{\Sp^1} f_1\left(x+z_1, y\right) f_2\left(x, y+z_2\right) d \sigma\left(z_1, z_2\right).
		\]
		Then there exists $\mathfrak{C}>0$ such that the following estimate holds,
		\[
		\left\|T\left(f_1, f_2\right)\right\|_{L^{1}\left(\mathbb{R}^2\right)} \lesssim \lambda^{-\mathfrak{C}}\left\|f_1\right\|_{L^2\left(\R^2\right)}\left\|f_2\right\|_{L^2\left(\R^2\right)},
		\]
		\sloppy for all test functions $f_1,f_2$ on $\R^2$ such that either
		\begin{enumerate}[i)]
			\item $\operatorname{supp} \hat{f}_1 \subset\left\{\left(\xi_1, \xi_2\right): \lambda \leq\left|\xi_1\right| \leq 2 \lambda\right\}$ and $\operatorname{supp} \hat{f}_2 \subset\left\{\left(\xi_1, \xi_2\right):\left|\xi_2\right| \leq 2 \lambda\right\}$, or
			\item $\operatorname{supp} \hat{f}_1 \subset\left\{\left(\xi_1, \xi_2\right): \left|\xi_1\right| \leq 2 \lambda\right\}$ and $\operatorname{supp} \hat{f}_2 \subset\left\{\left(\xi_1, \xi_2\right):\lambda\leq\left|\xi_2\right| \leq 2 \lambda\right\}$.
		\end{enumerate}
	\end{theorem}
	We refer the reader to \Cref{figure1} and \Cref{figure2} for the plan of proofs of \Cref{lacunary} and \Cref{Trilinearsmoothing} respectively. \Cref{full} is proved in \Cref{sec:prooffull} and \Cref{lacunary} is proved in \Cref{sec:prooflacunary}. \Cref{fiberwise:sec} is devoted to fiberwise bilinear Calder\'on-Zygmund theory. The proof of the trilinear smoothing inequality \Cref{Trilinearsmoothing} is carried out in \Cref{sec:Trilinearsmoothing} and \Cref{sublevel:sec} is devoted to a sublevel set estimate. In \Cref{sec:examples}, we discuss examples regarding necessary conditions for the boundedness of the bilinear spherical average $\mathfrak{A}_t$.
	\section{Full maximal function: Proof of \Cref{full}}\label{sec:prooffull} The proof of \Cref{full} in dimension $d\geq 2$ is based on the modified slicing argument for the spherical averages, see~\cite{BCSS1, JeongLee} for details. In the case of dimension $d=1$, the slicing argument is not valid. In this case, we estimate the circular averages in a more direct manner. We write down the integral over circle in parametric form and then decompose the domain of integral dyadically. We prove suitable estimates on each piece of the decomposition to complete the proof. We require the following notation in order to present a proof of \Cref{full}. 
	
	Consider the following linear averaging operators and associated maximal functions. These show up naturally in the analysis of bilinear spherical maximal functions. 
	\begin{align*}
		A_t^{(1)}f(x,y)&=\int_{\Sp^{d-1}}f(x-tz,y)\;d\sigma(z),\\
		A_t^{(2)}f(x,y)&=\int_{\Sp^{d-1}}f(x,y-tz)\;d\sigma(z),\\
		S^{(1)}f(x,y)&=\sup_{t>0}\;\left|\int_{\Sp^{d-1}}f(x-tz,y)\;d\sigma(z)\right|,\\
		S^{(2)}f(x,y)&=\sup_{t>0}\;\left|\int_{\Sp^{d-1}}f(x,y-tz)\;d\sigma(z)\right|.
	\end{align*}
	For $1\leq r\leq\infty$ and $l=1,2$, we define
	\begin{align*}
		M^{(1)}_rf(x,y)&=\sup_{t>0}\;\left(\int_{B(0,1)}|f(x-tz,y)|^r\;dz\right)^\frac{1}{r},\\
		M^{(2)}_rf(x,y)&=\sup_{t>0}\;\left(\int_{B(0,1)}|f(x,y-tz)|^r\;dz\right)^\frac{1}{r}.
	\end{align*}
	Further, in the case of dimensions $d\geq 2$, we require intermediary operators $A^{r,(l)}_*,\;1\leq r\leq\infty,\;l=1,2,$ defined by
	\[A^{r,(l)}_*f(x,y)=\sup_{k\in\Z}\left\|A_{2^kt}^{(l)}f(x,y)\right\|_{L^r([1,2],t^{d-1}dt)}.\]
	We have the following lemma concerning the $L^p-$estimates for the operators $A^{r,(l)}_*,\;1\leq r\leq\infty,\;l=1,2.$ The proof of this lemma is a direct consequence of Theorem 1.6 in \cite{BCSS1} applied in a fiberwise manner.
	\begin{lemma}\label{Arlinear}
		Let $1\leq r\leq\infty$, $d\geq2$, and $l=1,2$. Then, the operator $A^{r,(l)}_*$ is bounded from  $L^p(\R^{2d})$ to itself for $p>\frac{dr}{dr-r+1}$. Moreover $A^{r,(l)}_*$ is of restricted weak type $\left(\frac{dr}{dr-r+1},\frac{dr}{dr-r+1}\right)$ for $1\leq r<\infty$ and $A^{\infty,(l)}_*$ is of restricted weak type $\left(\frac{d}{d-1},\frac{d}{d-1}\right)$ for $d\geq3$.
	\end{lemma}
	
	Let us first consider the case of dimension $d\geq 2.$ We apply the modified slicing arguments from \cite{BCSS1} to obtain
	\[\mathfrak{M}(f_1,f_2)(x,y)\lesssim A^{r,(1)}_*f_1(x,y)A^{r',(2)}_*f_2(x,y),\;1\leq r\leq \infty.\]
	The case of $d\geq2$ follows by the  inequality above, \Cref{Arlinear} and H\"older's inequality.

	We now discuss the proof \Cref{full} for $d=1$.	We parameterize the circle $\Sp^1$ and use the symmetry to obtain
	\begin{align*}
		&\mathfrak{M}(f_1,f_2)(x,y)\\&\lesssim\sup_{t>0}\left|\int_{z=0}^{\frac{1}{2}}f_1(x-tz,y)f_2(x,y-t\sqrt{1-z^2})\;\frac{dz}{\sqrt{1-z^2}}\right|+\text{ similar terms}
	\end{align*}
	By decomposing the interval $[0,\frac{1}{\sqrt{2}}]$ into dyadic sub-intervals, we have
	\[\mathfrak{M}(f_1,f_2)(x)\lesssim\sum_{k=1}^{\infty}\mathfrak{M}_k(f_1,f_2)(x),\]
	where the operator $\mathfrak{M}_k$ is defined by
	\begin{align*}
		\mathfrak{M}_k(f_1,f_2)(x):&=\sup_{t>0}\int_{2^{-k-1}}^{2^{-k}}|f_1(x-tz,y)f_2(x,y-t\sqrt{1-z^2})|\;dz.
	\end{align*}
	By an application of H\"older's inequality, we have
	\begin{align}\label{T_k}
		&\mathfrak{M}_k(f_1,f_2)(x)\nonumber\\&\lesssim\min\;\left\{2^{\frac{k}{3}}M_3^{(1)}f_1(x,y)M_{\frac{3}{2}}^{(2)}f_2(x,y),2^{-\frac{k}{3}}M_\frac{3}{2}^{(1)}f_1(x,y)M_3^{(2)}f_2(x,y)\right\}.
	\end{align}
	Now, the restricted weak type $(2,2,1)$ inequality follows by the boundedness of the operators $M_r^{(l)}$ and summing in $k$. For details, we refer to \cite{BCSS1}. The other restricted weak type inequalities follows similarly. Finally, we conclude the desired strong type estimates by interpolation.
	\qed
	\section{Lacunary maximal function : Proof of \Cref{lacunary}} \label{sec:prooflacunary} The proof of \Cref{lacunary} constitutes the main body of 
	the paper. In this section, we provide the main steps required in the proof. Proofs of some of the major 
	intermediary steps will be presented in separate sections to make the exposition more streamlined for the interested reader. 
	The proof is motivated by the methods developed in~\cite{ChristZhou,CG, CDR, BCSS1} with suitable modifications 
	required to address the case of bilinear maximal function $\mathfrak{M}_{lac}$. We refer the reader to \Cref{figure1} for a brief outline of the proof of \Cref{lacunary}. 
	
	First, we decompose the operator using smooth Fourier projections of each of the two functions in a fiberwise manner. For, let $\phi\in C_c^\infty(\R^d)$ be a radial function supported on the ball $B(0,2)$ such that $\phi(x)=1,\;x\in B(0,1)$ and $0\leq\phi\leq1$. Let $\psi(x)=\phi(x)-\phi(2x)$. For $k\in\Z$, we define the sequence of functions $\phi_k$ and $\psi_k$ as $\phi_k(x)=\phi(2^{-k}x)$ and $\psi_k(x)=\psi(2^{-k}x)$. For a function $f\in\mathcal{S}(\R^{2d})$, we denote $\Theta_k^{(l)},\Delta_k^{(l)},\;l=1,2$ as the smooth Fourier projection operators given by
	\begin{align*}
		(\Theta_k^{(1)}f)^{\wedge}(\xi,y)=\phi_k(\xi)\widehat{f}(\xi,y),\;&\;(\Theta_k^{(2)}f)^{\wedge}(x,\eta)=\phi_k(\eta)\widehat{f}(x,\eta),\\
		(\Delta_k^{(1)}f)^{\wedge}(\xi,y)=\psi_k(\xi)\widehat{f}(\xi,y),\;&\;(\Delta_k^{(2)}f)^{\wedge}(x,\eta)=\psi_k(\eta)\widehat{f}(x,\eta).
	\end{align*}
	\vspace{0.5cm}
	\begin{figure}[H]
		\begin{center}
			\resizebox{\textwidth}{!}{
				\begin{tikzpicture}[node distance=2cm]
					\node (start) [startstop] {Bilinear Lacunary maximal function};
					\node (dec) [process, below of=start] {Littlewood-Paley Decomposition \eqref{decomposelacunary}};
					\node (low) [process, below right of=dec, xshift=-5cm] {Terms with low frequencies (\Cref{HLmax})};
					\node (high) [process, below left of=dec, xshift=5cm] {High frequency terms};
					\node (ppp) [process, below left of=high, xshift=-7cm, yshift=-1cm] {Weak $\left({\mathsf{p}_d},{\mathsf{p}_d},\frac{\mathsf{p}_d}{2}\right)$-Estimate \eqref{estimateGrowth}};
					\node (fiberwise) [process, below of=ppp] {Fiberwise Calder\'on-Zygmund Decomposition (\Cref{fiberwise:sec})};
					\node (singleavg) [process, below of=fiberwise] {$\left({\mathsf{p}_d},{\mathsf{p}_d},\frac{\mathsf{p}_d}{2}\right)$-bound for single average (\Cref{singleAvg})};
					\node (221) [process, below of=high, yshift=-1cm] {$(2,2,1)$-Estimate \eqref{estimateDecay}};
					\node (trilinear) [process, below of=221, yshift=-3cm] {Trilinear Smoothing Inequality (\Cref{Trilinearsmoothing})};
					\node (slicing) [process, below left of=221, xshift=-0.5cm, yshift=-1cm] {Slicing Argument};
					\draw [arrow] (start) -- (dec);
					\draw [arrow] (dec) -- (low);
					\draw [arrow] (dec) -- (high);
					\draw [arrow] (high) -- (ppp);
					\draw [arrow] (high) -- (221);
					\draw [arrow] (ppp) -- (fiberwise);
					\draw [arrow] (fiberwise) -- (singleavg);
					\draw [arrow] (221) -- node[anchor=west] {$d=1$} (trilinear);
					\draw [arrow] (221) -- node[anchor=east] {$d\geq2$} (slicing);
				\end{tikzpicture}
			}
		\end{center}
		\caption{\label{figure1} Sketch of proof of \Cref{lacunary}}
	\end{figure}
	
	By decomposing the functions $f_1$ and $f_2$ into smooth Littlewood-Paley pieces defined as above, we have
	\begin{align*}
		\mathfrak{M}_{lac}(f_1,f_2)\leq&\sup_{j\in\Z}\left|\mathfrak{A}_{2^j}(\Theta_j^{(1)}f_1,\Theta_j^{(2)}f_2)\right|+\sup_{j\in\Z}\left|\mathfrak{A}_{2^j}((I-\Theta_j^{(1)})f_1,\Theta_j^{(2)}f_2)\right|\\
		+&\sup_{j\in\Z}\left|\mathfrak{A}_{2^j}(\Theta_j^{(1)}f_1,(I-\Theta_j^{(2)})f_2)\right|+\sum_{{k}=(k_1,k_2)\in\N^2}\sup_{j\in\Z}\left|T_{j,{k}}(f_1,f_2)\right|
	\end{align*}
	\begin{gather}\label{decomposelacunary}
		\begin{aligned}
			\leq&3\sup_{j\in\Z}\left|\mathfrak{A}_{2^j}(\Theta_j^{(1)}f_1,\Theta_j^{(2)}f_2)\right|+\sup_{j\in\Z}\left|\mathfrak{A}_{2^j}(f_1,\Theta_j^{(2)}f_2)\right|\\
			+&\sup_{j\in\Z}\left|\mathfrak{A}_{2^j}(\Theta_j^{(1)}f_1,f_2)\right|+\sum_{{k}\in\N^2}\sup_{j\in\Z}\left|T_{j,{k}}(f_1,f_2)\right|,
		\end{aligned}
	\end{gather}
	where the operator $T_{j,{k}}$ is defined as
	\[T_{j,{k}}(f_1,f_2)(x,y):=\mathfrak{A}_{2^{-j}}(\Delta_{j+k_1}^{(1)}f_1,\Delta_{j+k_2}^{(2)}f_2)(x,y).\]
	The desired $L^p$-estimates of the first three terms in \eqref{decomposelacunary} follows from the following lemma.
	\begin{lemma}\label{HLmax}
		Let $d\geq1$, $1<p_1,p_2\leq\infty$, and $\frac{1}{p}=\frac{1}{p_1}+\frac{1}{p_2}$. Then the maximal operators
		\[\sup\limits_{j\in\Z}\left|\mathfrak{A}_{2^{-j}}(\Theta_j^{(1)}f_1,f_2)\right|,~\sup\limits_{j\in\Z}\left|\mathfrak{A}_{2^{-j}}(\Theta_j^{(1)}f_1,\Theta_j^{(2)}f_2)\right|,\text{ and }\sup\limits_{j\in\Z}\left|\mathfrak{A}_{2^{-j}}(f_1,\Theta_j^{(2)}f_2)\right|\]
		are bounded from $L^{p_1}(\R^d)\times L^{p_2}(\R^d)$ to $L^p(\R^d)$.
	\end{lemma}
	\subsection*{Proof of \Cref{HLmax}}	First, observe that the fiberwise convolution with $\Theta_j^{(l)}$ can be controlled by the Hardy-Littlewood maximal function in respective variable. For, $\phi\in\mathscr S(\R^d)$ and $|z_1|\leq 1$, we get that
	\begin{align*}		|\Theta_j^{(1)}f_1(x-2^{-j}z_1,y)|
		&=\left|\int_{\R^d} f_1(z,y)2^{jd}\phi(2^{j}(x-2^{-j}z_1-z)~dz\right|\\
		&\lesssim \left|\int_{\R^d} \frac{2^{jd}f_1(z,y)}{(1+|2^{j}(x-2^{-j}z_1-z)|)^{N}}dz\right|\\
		&\lesssim  M^{(1)}f_1(x,y).
	\end{align*}
	Now, we consider the case of $d=1$. Note that we have
	\begin{align*}
		&\sup_{j\in\mathbb{Z}}\left|\mathfrak{A}_{2^{-j}}(\Theta_j^{(1)}f_1,f_2)(x,y)\right|\\&\leq\sup_{j\in\mathbb{Z}}\sup_{|z_1|\leq1}|\Theta_j^{(1)}f_1(x-2^{-j}z_1,y)|\int_{\mathbb{S}^{1}}|f_2(x,y-2^{-j}z_2)|d\sigma(z_1,z_2).
	\end{align*}
	By symmetry and a change of variable, it is enough to consider
	\begin{align*}
		\sup_{j\in\mathbb{Z}}\int_0^1|f_2(x,y-2^{-j}t)|\frac{dt}{\sqrt{1-t^2}}&\simeq\sup_{j\in\mathbb{Z}}\sum_{k=1}^\infty2^\frac{k}{2}\int_{\sqrt{1-2^{-k+1}}}^{\sqrt{1-2^{-k}}}|f_2(x,y-2^{-j}t)|~dt\\
		&=\sum_{k=1}^\infty2^{-\frac{k}{2}}\sup_{j\in\mathbb{Z}}\frac{1}{|I_k|}\int_{I_k}|f_2(x,y-2^{-j}t)|~dt,
	\end{align*}
	where $I_k=[\sqrt{1-2^{-k+1}},\sqrt{1-2^{-k}}]$. Using the $L^p$-boundedness properties of shifted dyadic maximal function, we can get that
	\[\Big\|\sup_{j\in\mathbb{Z}}\Big|\mathfrak{A}_{2^{-j}}(\Theta_j^{(1)}f_1,f_2)\Big|\Big\|_p\lesssim\|f_1\|_{p_1}\|f_2\|_{p_2}\]
	when $1<p_1,p_2\leq\infty$ and $\frac{1}{p}=\frac{1}{p_1}+\frac{1}{p_2}$.
	
	Next, for $d\geq2$, a slicing argument gives us
	\begin{align*}
		&\sup_{j\in\mathbb{Z}}\left|\mathfrak{A}_{2^{-j}}(\Theta_j^{(1)}f_1,f_2)(x,y)\right|\\&\leq \sup_{j\in\mathbb{Z}}\sup_{|z_1|\leq1}|\Theta_j^{(1)}f_1(x-2^{-j}z_1,y)|\int_{\mathbb{S}^{2d-1}}|f_2(x,y-2^{-j}z_2)|d\sigma(z_1,z_2)\\
		&\lesssim M^{(1)}f_1(x,y)\sup_{j\in\mathbb{Z}}\int_{B^{d}(0,1)}|f_2(x,y-2^{-j}z_2)|(1-|z_2|^{2})^{(d-2)/2}\int_{\mathbb{S}^{d-1}}d\sigma(z_1)~dz_2\\
		&\leq M^{(1)}f_1(x,y)\sup_{j\in\mathbb{Z}}\int_{B^{d}(0,a)}|f_2(x,y-2^{-j}z_2)|~dz_2\\
		&\lesssim M^{(1)}f_1(x,y)M^{(2)}f_2(x,y).
	\end{align*}
	The $L^p$-estimates of $M^{(j)}$ for $j=1,2$ proves the desired result. Similarly, we can obtain the $L^p$-estimates for the remaining two operators in \Cref{HLmax}. \qed
	
	Continuing with the proof of \Cref{lacunary}, note that we need to establish 
	appropriate $L^p$-estimates of the fourth term 
	$$\sum_{{k}\in\N^2}\sup_{j\in\Z}\left|T_{j,{k}}(f_1,f_2)\right|$$ 
	in the~\Cref{decomposelacunary}. The desired estimate of this term is a consequence of the following two key propositions. 
	The first proposition is a weak type $\left({\mathsf{p}_d},{\mathsf{p}_d},\frac{\mathsf{p}_d}{2}\right)$-estimate 
	with a quadratic growth on the scale ${k}$. Indeed, we have
	\begin{proposition}\label{estimateGrowth}
		Let ${k}\in\N^2$. Then, for $f_1,f_2\in\mathcal{S}(\R^{2d})$, we have
		\[\Big\|\sup_{j\in\Z}|T_{j,{k}}(f_1,f_2)|\Big\|_{L^{\frac{\mathsf{p}_d}{2},\infty}}\lesssim |{k}|^{2}\|f_1\|_{L^{{\mathsf{p}_d}}}\|f_2\|_{L^{{\mathsf{p}_d}}}.\]
	\end{proposition}
	The second result provides a strong type $(2,2,1)$-estimate with 
	exponential decay in the parameter ${k}$. In particular, we have 
	\begin{proposition}\label{estimateDecay}
		There exists $\mathfrak{C}>0$ such that for $j\in\Z,\;{k}\in\N^2$ and $f_1,f_2\in\mathscr{S}(\R^{2d})$, we have
		\[\Big\|\sum_{j\in\Z}|T_{j,{k}}(f_1,f_2)|\Big\|_{L^1}\lesssim 2^{-\mathfrak{C}|{k}|}\|f_1\|_{L^2}\|f_2\|_{L^2}.\]
	\end{proposition} 
	Observe that the $L^p$-estimates of the fourth term 
	$\sum_{{k}\in\N^2}\sup_{j\in\Z}\left|T_{j,{k}}(f_1,f_2)\right|$ in \Cref{decomposelacunary} 
	for the entire range of exponents, as claimed in \Cref{lacunary}, follow
	by interpolating between \Cref{estimateGrowth} and \Cref{estimateDecay} and finally summing over the 
	parameter ${k}$. Therefore, the proof of \Cref{lacunary} is complete under the assumption that 
	\Cref{estimateGrowth} and \Cref{estimateDecay} hold. 
	\subsection*{About proof of \Cref{estimateGrowth}} The proof of \Cref{estimateGrowth} is based on 
	the bilinear Calder\'{o}n-Zygmund theory adapted to fiberwise decomposition of the functions. We also require $L^p$-estimates 
	for the single scale averaging operator $\mathfrak{A}_1,$ see \Cref{singleAvg}, in order to apply the fiberwise bilinear Calder\'{o}n-Zygmund theory. 
	Since this is one of the major steps in the proof of \Cref{lacunary}, we present the proof of \Cref{estimateGrowth} in \Cref{fiberwise:sec}. 
	\subsection*{Proof of \Cref{estimateDecay}}
	The proof of \Cref{estimateDecay} for the case of dimensions $d\geq2$ is obtained by the modified slicing argument. We present it below.  
	
	By an application of modified slicing argument and Cauchy-Schwarz inequality, we have
	\begin{align*}
		&\|T_{j,{k}}(f_1,f_2)\|_1\\
		\leq&\int\limits_{\R^{2d}}\int\limits_{0}^1 \Big|A_r^{(1)}\Delta_{j+k_1}^{(1)}\Delta_{j+k_1+1}^{(1)}f_1(x,y) A_{\sqrt{1-r^2}}^{(2)}\Delta_{j+k_2}^{(2)}\Delta_{j+k_2+1}^{(2)}f_2(x,y)\Big|\\
		&\hspace{9mm}\times r^{d-1}(1-r^2)^\frac{d-2}{2}\;drdxdy\\
		\leq&\int\limits_{0}^1 \Big\|A_r^{(1)}\Delta_{j+k_1}^{(1)}\Delta_{j+k_1+1}^{(1)}f_1\Big\|_2 \Big\|A_{\sqrt{1-r^2}}^{(2)}\Delta_{j+k_2}^{(2)}\Delta_{j+k_2+1}^{(2)}f_2\Big\|_2r^{d-1}(1-r^2)^\frac{d-2}{2}drdxdy\\
		\lesssim&2^{-(k_1+k_2)(\frac{d-1}{2})}\|\Delta_{j+k_1+1}^{(1)}f_1\|_2\|\Delta_{j+k_2+1}^{(2)}f_2\|_2\int_0^1r^{\frac{d-1}{2}}(1-r^2)^{\frac{d-3}{4}}\;dr\\
		\lesssim&2^{-|{k}|(\frac{d-1}{2})}\|\Delta_{j+k_1+1}^{(1)}f_1\|_2\|\Delta_{j+k_2+1}^{(2)}f_2\|_2.
	\end{align*}
	In the fourth line of the expression above we have used the Plancherel theorem in $\R^d$ for each fiber. 
	
	We conclude the proof of \Cref{estimateDecay} for $d\geq 2$ by summing over $j\in\Z$ and using the fact that the Fourier supports of functions $\Delta_{j+1}^{(l)}f_l,j\in\Z$ have bounded overlap for each $l=1,2$.
	
	Next, we discuss the case of dimension $d=1$. This case is subtle. Indeed, it forms the core of the paper. The proof is motivated by the techniques developed in~\cite{CG, CDR, ChristZhou}. The main step in the proof of \Cref{lacunary} for $d=1$ is the trilinear smoothing inequality, which is the content of in \Cref{Trilinearsmoothing}. The rest of the paper is devoted to establishing \Cref{Trilinearsmoothing}. Further, in the proof of \Cref{Trilinearsmoothing} the following theorem for sublevel sets is crucial. 
	\begin{theorem}\label{Sublevelsetestimate}
		Let $I=[\delta,2\delta]$, $K \subseteq \mathbb{R}^{2} \times I$ be a compact set. Let $\alpha,\beta:\R^2\to\R$ be measurable functions such that either
		\begin{enumerate}[(i)]
			\item $|\alpha(t)|\sim1$ and $|\beta(t)|\lesssim1$ for all $t\in I$, or
			\item $|\alpha(t)|\lesssim1$ and $|\beta(t)|\sim1$ for all $t\in I$.
		\end{enumerate}
		Then, for all $\epsilon\in(0,1]$, we have
		\[|\{(x,y,t)\in K:\;|\alpha(x+\cos t,y)\sin t+\beta(x,y+\sin t)\cos t|\lesssim\delta\epsilon\}|\lesssim\delta\epsilon^{\frac{1}{35}}.\]
	\end{theorem}
	\Cref{sublevel:sec} is devoted to prove \Cref{Sublevelsetestimate}. 
	\section{Fiberwise bilinear Calder\'{o}n-Zygmund theory: Proof of \Cref{estimateGrowth}}\label{fiberwise:sec}
	First, we establish the following $L^p$-estimate for the single scale spherical average $\mathfrak{A}_{1}$. 
	\begin{proposition}\label{singleAvg}
		Let $d\geq1$. We have
		\[\|\mathfrak{A}_{1}(f_1,f_2)\|_{L^\frac{\mathsf{p}_d}{2}(\R^{2d})}\lesssim \|f_1\|_{L^{\mathsf{p}_d}(\R^{2d})}\|f_2\|_{L^{\mathsf{p}_d}(\R^{2d})},\;\text{for all}\;f_1,f_2\in\mathcal{S}(\R^{2d}).\]
	\end{proposition}
	\subsection*{Proof of \Cref{singleAvg}}
	Note that using the localisation trick, see  \cite[Proposition 4.1]{IPS}, it is enough	to prove $L^{\mathsf{p}_d}(\R^{2d})\times L^{\mathsf{p}_d}(\R^{2d})\to L^1(\R^{2d})$-boundedness of the operator $\mathfrak{A}_1$. Consider 
	\begin{align*}
		\|\mathfrak{A}_1(f_1,f_2)\|_1&=\int_{\R^{2d}}\Big|\int_{\mathbb S^{2d-1}}f_1(x+z_1,y)f_2(x,y+z_2)\;d\sigma(z_1,z_2)\Big|dxdy\\
		&\leq\int_{\mathbb S^{2d-1}}\int_{\R^{2d}}|f_1(x+z_1,y)f_2(x,y+z_2)|\;dxdyd\sigma(z_1,z_2)\\
		&=\int_{\mathbb S^{2d-1}}\int_{\R^{2d}}|f_1(x,y)f_2(x-z_1,y+z_2)|\;dxdyd\sigma(z_1,z_2)\\
		&=\int_{\R^{2d}}|f_1(x,y)|\int_{\mathbb S^{2d-1}}|f_2(x-z_1,y+z_2)|\;d\sigma(z_1,z_2)dxdy\\
		&\lesssim\|f_1\|_{{\mathsf{p}_d}}\Big\|\int_{\mathbb S^{2d-1}}|f_2(\cdot-z_1,\cdot+z_2)|\;d\sigma(z_1,z_2)\Big\|_{2d+1}\\
		&\lesssim\|f_1\|_{{\mathsf{p}_d}}\|f_2\|_{{\mathsf{p}_d}}.
	\end{align*}
	\qed

	Next, we will employ a fiberwise Calder\'on-Zygmund decomposition of functions in $L^p-$spaces based on the decomposition obtained in \cite{Bernicot}. 
	\begin{lemma}\label{fiberwiseCZ}
		Let $1\leq p<\infty$ and $f\in L^p(\R^{2d})$. For $\alpha>0$ and fixed $y\in \R^d$, we apply the Calder\'on-Zygmund decomposition at the scale $\alpha$ and obtain $f(\cdot,y)=g(\cdot,y)+\sum\limits_{i}h_i(\cdot,y)$, where $h_i(\cdot,y)$ is supported on dyadic cubes $Q_{i,y}\subset\R^d$ and
		\begin{enumerate}[(a)]
			\item $\|g(\cdot,y)\|_{L^p(\R^d)}\leq\|f(\cdot,y)\|_{L^p(\R^d)}$ and $\|g(\cdot,y)\|_{L^\infty(\R^d)}\lesssim\alpha$.
			\item For each $i$, $\|h_i(\cdot,y)\|_{L^p(\R^d)}^p\lesssim\alpha^p|Q_{i,y}|$ and $\int_{Q_{i,y}}h_i(x,y)~dx=0$.
			\item $\sum\limits_i|Q_{i,y}|\lesssim\alpha^{-p}\|f(\cdot,y)\|_{L^p(\R^d)}^p$.
		\end{enumerate}
	\end{lemma}
	Now, we argue for the proof of \Cref{estimateGrowth}. Without loss of generality, we assume $\|f_1\|_{{\mathsf{p}_d}}=\|f_2\|_{{\mathsf{p}_d}}=1$. We apply \Cref{fiberwiseCZ} to $f_1,f_2\in L^{\mathsf{p}_d}$ in $x,y$ variables respectively at the scale $c_0\alpha^{\frac{1}{2}},$ for some $c_0>0$ to be determined later and obtain $f_l=g_l+h_l,\;l=1,2$, where
	\begin{itemize}
		\item $\|g_1(\cdot,y)\|_{L^\infty(\R^d)}\lesssim\alpha^\frac{1}{2}$ and $h_1=\sum\limits_{\beta_1}h_{1,\beta_1}(\cdot,y)$, where $h_{1,\beta_1}(\cdot,y)$ is supported on dyadic cubes $Q_{\beta_1,y}\subset\R^d$, with
		\begin{align*}
			&\|h_{1,\beta_1}(\cdot,y)\|_{L^{{\mathsf{p}_d}}(\R^d)}^{\mathsf{p}_d}\lesssim\alpha^\frac{\mathsf{p}_d}{2}|Q_{\beta_1,y}|,\\
			&\sum\limits_{\beta_1}|Q_{\beta_1,y}|\lesssim\alpha^{-\frac{\mathsf{p}_d}{2}}\|f_1(\cdot,y)\|_{L^{{\mathsf{p}_d}}(\R^d)}^{\mathsf{p}_d},\\
			&\int_{Q_{\beta_1,y}}h_{1,\beta_1}(x,y)~dx=0.
		\end{align*}
		\item $\|g_2(x,\cdot)\|_{L^\infty(\R^d)}\lesssim\alpha^\frac{1}{2}$ and $h_2=\sum\limits_{\beta_2}h_{2,\beta_2}(x,\cdot)$, where $h_{2,\beta_2}(x,\cdot)$ is supported on dyadic cubes $Q_{\beta_2,x}\subset\R^d$, with
		\begin{align*}
			&\|h_{2,\beta_2}(x,\cdot)\|_{L^{{\mathsf{p}_d}}(\R^d)}^{\mathsf{p}_d}\lesssim\alpha^\frac{\mathsf{p}_d}{2}|Q_{\beta_2,x}|,\\
			&\sum\limits_{\beta_2}|Q_{\beta_2,x}|\lesssim\alpha^{-\frac{\mathsf{p}_d}{2}}\|f_2(x,\cdot)\|_{L^{{\mathsf{p}_d}}(\R^d)}^{\mathsf{p}_d},\\
			&\int_{Q_{\beta_2,x}}h_{2,\beta_2}(x,y)~dy=0.
		\end{align*}
	\end{itemize}
	\subsection*{Contribution of $(g_1,g_2)$}
	Since $\|g_l\|_\infty\leq c_0\alpha^{\frac{1}{2}}$, we have $\left\|\sup\limits_{j\in\Z}|T_{j,k}(g_1,g_2)|\right\|_\infty\leq c_0^2\alpha$. We choose $c_0$ small so that $c_0^2\alpha\leq\frac{\alpha}{4}$.
	\subsection*{Contribution of $(g_1,h_2)$ and $(h_1,g_2)$}
	We note that $\|\Delta_{j+k_1}^{(1)}g_1\|_\infty\lesssim\alpha^\frac{1}{2}$ uniformly in $j$ and $k_1$. An argument similar to that in \Cref{HLmax}, we can see that $\left\|\sup\limits_{j\in\Z}|T_{j,k}(1,h_2)|\right\|_{{\mathsf{p}_d}}\leq \|h_2\|_{{\mathsf{p}_d}}$. Therefore
	\begin{align*}
		&\left|\left\{(x,y)\in\R^d\times\R^d:\sup_{j\in\Z}|T_{j,k}(g_1,h_2)(x,y)|>\alpha\right\}\right|\\
		\lesssim&\alpha^{-{\mathsf{p}_d}}\left\|\sup\limits_{j\in\Z}|T_{j,k}(g_1,h_2)|\right\|_{\mathsf{p}_d}^{\mathsf{p}_d}\\
		\lesssim&\alpha^{-{\mathsf{p}_d}}\alpha^\frac{\mathsf{p}_d}{2}\|f_2\|_{\mathsf{p}_d}^{\mathsf{p}_d}\\
		\leq& \alpha^{-\frac{\mathsf{p}_d}{2}}.
	\end{align*}
	Similarly, we can get the desired estimate for $(h_1,g_2)$ by interchanging the roles of the indices.
	
	\subsection*{Contribution of $(h_1,h_2)$} Let $Q^*_{\beta,y}$ be the concentric cube with $Q_{\beta,y}$ and measure $4^d|Q_{\beta,y}|$. Now, define
	\[E=\{(x,y):x\in\cup_{\beta_1}Q^*_{\beta_1,y}\}\cup\{(x,y):y\in\cup_{\beta_2}Q^*_{\beta_2,x}\}.\]
	We can see that
	\begin{align*}
		|E|&\leq\int_{\R^d}\sum_{\beta_1}|Q^*_{\beta_1,y}|dy+\int_{\R^d}\sum_{\beta_2}|Q^*_{\beta_2,x}|dx\\
		&\lesssim\alpha^{-\frac{\mathsf{p}_d}{2}}\int_{\R^d}\|f_1(\cdot,y)\|_{L^{{\mathsf{p}_d}}(\R^d)}^{\mathsf{p}_d}dy+\alpha^{-\frac{\mathsf{p}_d}{2}}\int_{\R^d}\|f_2(x,\cdot)\|_{L^{{\mathsf{p}_d}}(\R^d)}^{\mathsf{p}_d}dx\\
		&=\alpha^{-\frac{\mathsf{p}_d}{2}}\left(\|f_1\|_{L^{{\mathsf{p}_d}}(\R^{2d})}^{\mathsf{p}_d}+\|f_2\|_{L^{{\mathsf{p}_d}}(\R^{2d})}^{\mathsf{p}_d}\right)\lesssim\alpha^{-\frac{\mathsf{p}_d}{2}}.
	\end{align*}
	Thus, we need to show that
	\begin{equation}\label{weakbad}
		\int_{\R^{2d}\setminus E}\sup_{j\in\Z}\left|\mathfrak{A}_{2^j}(\Delta_{j+k_1}^{(1)}h_1,\Delta_{j+k_2}^{(2)}h_2)(x,y)\right|^\frac{\mathsf{p}_d}{2}dxdy\lesssim|{k}|^2.
	\end{equation}
	We define $h_1^{i_1}=\sum\limits_{|Q_{\beta_1,y}|=2^{-i_1}}h_{1,\beta_1}$ and $h_2^{i_2}=\sum\limits_{|Q_{\beta_2,x}|=2^{-i_2}}h_{2,\beta_2}$, write
	\begin{align*}
		&\sup_{j\in\Z}\left|\mathfrak{A}_{2^j}(\Delta_{j+k_1}^{(1)}h_1,\Delta_{j+k_2}^{(2)}h_2)(x,y)\right|^\frac{\mathsf{p}_d}{2}\\&=\sup_{j\in\Z}\left|\sum_{i_1\in\Z}\sum_{i_2\in\Z}\mathfrak{A}_{2^j}(\Delta_{j+k_1}^{(1)}h_1^{i_1},\Delta_{j+k_2}^{(2)}h_2^{i_2})(x,y)\right|^\frac{\mathsf{p}_d}{2}\\        &\leq\sum_{i_1\in\Z}\sum_{i_2\in\Z}\sum_{j\in\Z}\left|\mathfrak{A}_{2^j}(\Delta_{j+k_1}^{(1)}h_1^{i_1},\Delta_{j+k_2}^{(2)}h_2^{i_2})(x,y)\right|^\frac{\mathsf{p}_d}{2}.
	\end{align*}
	We have the following lemma stating $L^p-$bounds for terms in sums above. 
	\begin{lemma}
		The following bound holds uniformly in $j,{k},i_1,i_2$,
		\begin{align*}
			&\int_{\R^{2d}\setminus E}\left|\mathfrak{A}_{2^j}(\Delta_{j+k_1}^{(1)}h_1^{i_1},\Delta_{j+k_2}^{(2)}h_2^{i_2})(x,y)\right|^\frac{\mathsf{p}_d}{2}dxdy\\&\lesssim\min_{l=1,2}\min\{2^{i_l-j},2^{\frac{j+k_l-i_l}{2}},1\}\prod_{l=1}^2\|h_l^{i_l}\|_{L^{{\mathsf{p}_d}}(\R^{2d})}^\frac{\mathsf{p}_d}{2}.
		\end{align*}
	\end{lemma}
	\begin{proof}
		Using \Cref{singleAvg}, we can get that
		\begin{equation}\label{badloc}
			\int_{\R^{2d}\setminus E}\left|\mathfrak{A}_{2^j}(\Delta_{j+k_1}^{(1)}h_1^{i_1},\Delta_{j+k_2}^{(2)}h_2^{i_2})(x,y)\right|^\frac{\mathsf{p}_d}{2}dxdy\lesssim\prod_{l=1}^2\|\Delta_{j+k_l}^{(l)}h_l^{i_l}\|_{L^{{\mathsf{p}_d}}(\R^{2d})}^\frac{\mathsf{p}_d}{2}.
		\end{equation}
		First, note that the operator $\Delta_{j+k_l}$ is defined by convolution with an integrable function whose norm is independent of $l$. Thus, for all $1\leq p\leq\infty$, we have
		\begin{equation}\label{fiberLp}
			\|\Delta_{j+k_l}^{(l)}h_l^{i_l}\|_{L^{p}(\R^{2d})}\lesssim\|h_l^{i_l}\|_{L^{p}(\R^{2d})}.
		\end{equation}
		When $i_l>j+k_l$, using the moment condition $\int h_{l,\beta_l}(x,y)~dx=0$ and the fact that each $h_{l,\beta_l}$ is supported on an cube of length $2^{-i_l}$, we get
		\begin{equation}\label{fiberL1}
			\|\Delta_{j+k_l}^{(l)}h_l^{i_l}\|_{L^{1}(\R^{2d})}\lesssim2^{j+k_l-i_l}\|h_l^{i_l}\|_{L^{1}(\R^{2d})}.
		\end{equation}
		Interpolating estimates \eqref{fiberL1} and \eqref{fiberLp} for $p=\infty$, we get that
		\[\|\Delta_{j+k_l}^{(l)}h_l^{i_l}\|_{L^{{\mathsf{p}_d}}(\R^{2d})}\lesssim2^{\frac{(j+k_l-i_l)2d}{2d+1}}\|h_l^{i_l}\|_{L^{{\mathsf{p}_d}}(\R^{2d})}.\]
		
		Next, we consider the case when $i_1<j$ and $x\notin E$. Note that $x$ lies at a distance at least $2\cdot 2^{-i_1}$ from cubes $Q_{\beta_1,y}$ of measure $2^{-i_1}$. Since $i_1<j$, the distance of $x+2^{-j}z_1$ from cubes $Q_{\beta_1,y}$ is bigger than $2^{-i_1}$. Thus, we can write
		\[\mathfrak{A}_{2^{-j}}(\Delta_{j+k_1}^{(1)}h_1^{i_1},\Delta_{j+k_2}^{(2)}h_2^{i_2})(x,y)=\mathfrak{A}_{2^{-j}}(\widetilde{\psi}\ast^{(1)}h_1^{i_1},\Delta_{j+k_2}^{(2)}h_2^{i_2})(x,y),\]
		where $\widetilde{\psi}(x)=\psi_{j+k_1}(x)\chi_{|x|>2^{-i_1}}$ and satisfies the estimate for any $N$,
		\[\|\widetilde{\psi}\|_{L^1}=\|\psi_{j+k_1}\chi_{|x|>2^{-i_1}}\|_{L^1}\lesssim 2^{-(j+k_1-i_1)N}.\]
		Therefore, we obtain
		\begin{align*}
			\|\widetilde{\psi}\ast^{(1)}h_1^{i_1}\|_{L^{{\mathsf{p}_d}}(\R^{2d})}&\leq \|\widetilde{\psi}\|_{L^1(\R^d)}\|h_1^{i_1}\|_{L^{{\mathsf{p}_d}}(\R^{2d})}\\
			&\lesssim 2^{-(j+k_1-i_1)N}\|h_1^{i_1}\|_{L^{{\mathsf{p}_d}}(\R^{2d})}.
		\end{align*}
		Similarly when $i_2<j$, interchanging roles of $i_1$ and $i_2$, we get the estimate
		\[\|\widetilde{\psi}\ast^{(1)}h_1^{i_1}\|_{L^{{\mathsf{p}_d}}(\R^{2d})}\lesssim 2^{-(j+k_2-i_2)N}\|h_1^{i_1}\|_{L^{{\mathsf{p}_d}}(\R^{2d})}.\]
	\end{proof}
	To obtain the estimate \eqref{weakbad}, we sum in $j$ and obtain
	\begin{align*}
		&\int_{\R^{2d}\setminus E}\sup_{j\in\Z}\left|\mathfrak{A}_{2^j}(\Delta_{j+k_1}^{(1)}h_1,\Delta_{j+k_2}^{(2)}h_2)(x,y)\right|^\frac{\mathsf{p}_d}{2}dxdy\\
		\lesssim&\sum_{i_1\in\Z}\sum_{i_2\in\Z}|{k}|\min\{1,2^{\frac{-(|i_1-i_2|-|{k}|)}{2}}\}\|h_1^{i_1}\|^\frac{\mathsf{p}_d}{2}_{L^{{\mathsf{p}_d}}}\|h_2^{i_2}\|_{L^{{\mathsf{p}_d}}}^\frac{\mathsf{p}_d}{2}\\
		\leq&|{k}|\left(\sum_{i_1,i_2\in\Z}\min\{1,2^{\frac{-(|i_1-i_2|-|{k}|)}{2}}\}\|h_1^{i_1}\|^{\mathsf{p}_d}_{L^{{\mathsf{p}_d}}}\right)^{\frac{1}{2}}\\
		&\hspace{5mm}\times\left(\sum_{i_1,i_2\in\Z}\min\{1,2^{\frac{-(|i_1-i_2|-|{k}|)}{2}}\}\|h_2^{i_2}\|^{\mathsf{p}_d}_{L^{{\mathsf{p}_d}}}\right)^{\frac{1}{2}}\\
		\lesssim&|{k}|^2\left(\sum_{i_1\in\Z}\|h_1^{i_1}\|^{\mathsf{p}_d}_{L^{{\mathsf{p}_d}}}\right)^{\frac{1}{2}}\left(\sum_{i_2\in\Z}\|h_2^{i_2}\|^{\mathsf{p}_d}_{L^{{\mathsf{p}_d}}}\right)^{\frac{1}{2}}\\
		\leq&|{k}|^2.
	\end{align*}
	\qed

	\section{Trilinear Smoothing Estimate: Proof of \Cref{Trilinearsmoothing}} \label{sec:Trilinearsmoothing}
	We have included a brief diagram, see~\Cref{figure2}, describing a brief outline of proof of \Cref{Trilinearsmoothing} for reader's convenience.  
	\subsection{Localization of the operator to isolate the degeneracies of $\cos t$ and $\sin t$}\label{sec:localisolation}
	By the $2 \pi$-periodic parametrization 
	$\{(\cos t, \sin t): t \in[0,2 \pi)\}$ 
	of the unit sphere $\Sp^1$, it is enough to prove the local smoothing estimate for the operator
	\[
	T_0\left(f_1, f_2\right)(x, y)
	=
	\int f_1(x+\cos t, y) f_2(x, y+\sin t) \zeta_0(t) d t \text {, where }
	\]
	$\zeta_0$ is a smooth function on $\mathbb{R}$ supported in a compact neighbourhood of any one of points belonging to the critical set $\left\{0, \frac{\pi}{2}, \pi, \frac{3 \pi}{2}\right\}$ i.e. the roots of $\cos t$ and $\sin t$.
	
	\subsection{Localization of the operator in the space variables $(x, y)$}\label{sec:spacelocal}
	In order to prove the trilinear smoothing inequality, it is enough to prove the estimate for the local operator
	\begin{equation}\label{Defn:Tloc}
		T_{loc}\left(f_1, f_2\right)(x, y)=\int f_1(x+\cos t, y) f_2(x, y+\sin t) \zeta(x,y,t) d t,
	\end{equation}
	where $\zeta$ is a smooth compactly supported function on $\mathbb{R}^2 \times[0,2\pi)$. Without loss of generality, we can assume that $\zeta$ is supported in $\mathbb{R}^2 \times N(0)$, where $N(0)$ is an open neighbourhood of $t=0$.
	\begin{figure}[H]
		\begin{center}
			\resizebox{\textwidth}{!}{
				\begin{tikzpicture}[node distance=2cm]
					\node (trilinear) [startstop] {Trilinear Smoothing Inequality (\Cref{Trilinearsmoothing})};
					\node (isolation) [process, below of=trilinear] {Isolating the degeneracies of $\cos t$ and $\sin t$ (\Cref{sec:localisolation})};
					\node (spatial) [process, below of=isolation] {Spatial decomposition of the operator (\Cref{sec:spacelocal})};
					\node (inftyinfty1) [process, below of=spatial] {Reduction to an $(\infty,\infty,1)$-estimate (\Cref{sec:inftyinfty1})};
					\node (degeneracy) [process, below of=inftyinfty1] {Decomposition of the operator based on the degeneracy of the curve (\Cref{sec:degeneracy})};
					\node (pruning) [process, below of=degeneracy] {Frequency pruning and decomposition of the functions (\Cref{sec:pruning})};
					\node (smallmd) [process, below right of=pruning, xshift=-5cm, yshift=-1cm] {Terms with small multiplicative derivative (\Cref{lemma:oscillatory})};
					\node (sublevel) [process, below left of=pruning, xshift=5cm, yshift=-2.2cm] {Sublevel set estimates (\Cref{Sublevelsetestimate})};
					
					\draw [arrow] (trilinear) -- (isolation);
					\draw [arrow] (isolation) -- (spatial);
					\draw [arrow] (spatial) -- (inftyinfty1);
					\draw [arrow] (inftyinfty1) -- (degeneracy);
					\draw [arrow] (degeneracy) -- (pruning);
					\draw [arrow] (pruning) -- (smallmd);
					\draw [arrow] (pruning) -- node[anchor=west] {Reduction via \Cref{lemma:reductiontosublevel}} (sublevel);
				\end{tikzpicture}
			}
		\end{center}
		\caption{\label{figure2} Sketch of proof of \Cref{Trilinearsmoothing}}
	\end{figure}
	
	We now justify the above reduction. Let $\eta$ be a smooth non-negative function on $\mathbb{R}^2$ such that it is supported in $\left[-\frac{3}{4}, \frac{3}{4}\right]^2$
	and $\sum_{m\in\Z^2} \eta_m(x, y)=1$, where
	$\eta_m(x, y)=\eta((x, y)-m)$. We also consider a function $\widetilde{\eta}$ which is supported in $[-2,2]^2$ and takes value one on $[-1,1]^2$. Observe that
	\begin{align*}
		f_1(x,y)&=f_1 *_1 \widetilde{\psi}_\lambda(x,y)=\int_{\R}f_1(x-z,y)\widetilde{\psi}_\lambda(z)\;dz \quad \text{and}\\
		f_2(x,y)&=f_2 *_2 \widetilde{\phi}_\lambda(x,y)=\int_{\R}f_2(x,y-z)\widetilde{\phi}_\lambda(z)\;dz,
	\end{align*}
	where $\widehat{\widetilde{\psi}}_\lambda$ is a smooth function supported in $\left\{\frac{\lambda}{2} \leq|\xi_1| \leq 4 \lambda\right\}$ and $\widehat{\widetilde{\psi}}_\lambda=1$ on $\{\lambda \leq|\xi_1| \leq 2 \lambda\}$ and $\widehat{\widetilde{\phi}}_\lambda$ is a smooth function supported in $\left\{|\xi_2| \leq 4 \lambda\right\}$ and $\widehat{\widetilde{\phi}}_\lambda=1$ on $\{|\xi_2| \leq 2 \lambda\}$.
	We have,
	\[
	\begin{aligned}
		& \left\|T_0\left(f_1 *_1 \psi_\lambda, f_2 *_2 \phi_\lambda\right)\right\|_1 \\
		&\leq  \sum_{m \in \mathbb{Z}^2}\left\|T_{loc}\left(\widetilde{\eta}_m\left(f_1 *_1 \widetilde{\psi}_\lambda\right), \widetilde{\eta}_m\left(f_2 *_2 \widetilde{\phi}_\lambda\right)\right) \eta_m\right\|_1 \\
		&\leq  \sum_{m \in \Z^2} \left\| T_{loc}\left(\left(f_1 \widetilde{\eta}_m\right) *_1 \widetilde{\psi}_\lambda,\left(f_2 \widetilde{\eta}_m\right) *_2 \widetilde{\phi}_\lambda\right) \right\|_1 \\
		& +\sum_{m \in \Z^2}\left\|T_{loc}\left(\widetilde{\eta}_m\left(f_1 *_1 \widetilde{\psi}_\lambda\right)-\left(\widetilde{\eta}_m f_1\right) *_1 \widetilde{\psi}_\lambda, \widetilde{\eta}_m\left(f_2 *_2 \widetilde{\phi}_\lambda\right)\right) \eta_m\right\|_1 \\
		&+  \sum_{m \in \Z^2}\left\|T_{loc}\left(\widetilde{\eta}_m\left(f_1 *_1 \widetilde{\psi}_\lambda\right), \widetilde{\eta}_m\left(f_2 *_2 \widetilde{\phi}_\lambda\right)-\left(\widetilde{\eta}_m f_2\right) *_2 \widetilde{\phi}_\lambda\right) \eta_m\right\|_1 \\
		&= I_1+I_2+I_3.
	\end{aligned}
	\]
	
	The terms in $I_1$ are given by operators of the form \eqref{Defn:Tloc}, where $\zeta(x, y, t)=\widetilde{\eta}_m(x+\cos t, y) \widetilde{\eta}_m(x, y+\sin t) \zeta_0(t)$. Hence, the $L^2 \times L^2 \rightarrow L^{1}$-estimate for the local operator $T_{loc}$ and use of Cauchy-Schwarz inequality with the fact that $\sum_{m \in \Z^2} \widetilde{\eta}_m^2 \lesssim 1$ lead us to the desired estimate for the term $I_1$.
	
	The estimate for the commutator terms $I_2$ and $I_3$ are proved in a similar manner. Therefore, we indicate the proof for the term $I_2$ only. We observe, by mean value theorem, that
	\[
	\begin{aligned}
		& \left|\left(\widetilde{\eta}_m \left(f_1 *_1 \widetilde{\psi}_\lambda\right)-\left(\widetilde{\eta}_m f_1\right) *_1\widetilde{\psi}_1\right)(x, y)\right| \\
		\lesssim & \int\left|\widetilde{\eta}_m(x, y)-\widetilde{\eta}_m(u, y)\right|\left|f_1(u, y)\right|\left|\widetilde{\psi}_\lambda(x-u)\right| d u \\
		\lesssim & \lambda^{-1}\int \frac{\lambda|x-u|}{(1+\lambda|x-u|)^3} \quad\left|f_1(u, y)\right| d u \\
		\leq & \lambda^{-1}\left|f_1\right| *_1 \Psi_\lambda,
	\end{aligned}
	\]
	
	where $\Psi_\lambda(y)=\lambda(1+\lambda|y|)^2$. Therefore it follows that
	\[
	\begin{aligned}
		I_2 &\lesssim \lambda^{-1} \int_{\mathbb{R}^2} \int_{\mathbb{R}}\big(\left|f_1\right| *_1 \Psi_\lambda\big)(x+\cos t, y) \big(\left|f_2\right|*_2|\widetilde{\phi}_\lambda|\big)(x, y+\sin t)\\
		&\hspace{2.2cm}\times\sum_m\left|\eta_m(x, y)\right| \zeta_0(t)\;dtdxdy \\
		& \lesssim \lambda^{-1} \int_{\mathbb{R}}\big\|\left|f_1\right| *_1 \Psi_\lambda\big\|_2 \big\|\left|f_2\right| *_2 |\widetilde{\phi}_\lambda|\big\|_2 \ \zeta_0(t) \;dt \\
		& \lesssim \lambda^{-1}\left\|f_1\right\|_2 \left\|f_2\right\|_2.
	\end{aligned}
	\]
	
	\subsection{Reduction of the $(2,2,1)$-estimate to an $(\infty, \infty, 1)$-estimate}\label{sec:inftyinfty1}
	We reduce the $(2,2,1)$-estimate to proving an $(\infty, \infty, 1)$-estimate by using the $L^{\frac{3}{2}} \to L^3$ improving property for the linear spherical average. Indeed, we have
	\[
	\begin{aligned}
		\left\|T_{loc}\left(f_{1}, f_2\right)\right\|_1\approx & \int_{\mathbb{R}^2}\left|f_1(x, y)\right| \int_0^{2 \pi}\left|f_2(x-\cos t,y+\sin t)\right| d t d x d y \\
		\leq & \left\|f_1\right\|_{\frac{3}{2}}\left\|\int_0^{2 \pi} \mid f_2((x, y)-(\cos t, \sin t))|\;dt\right\|_3 \\
		\approx & \left\|f_1\right\|_{\frac{3}{2}} \left\|f_2\right\|_{\frac{3}{2}} .
	\end{aligned}
	\]
	Thus by Sobolev interpolation and the previous estimate, it remains to prove
	\[\left\|T_{loc}\left(f_1, f_2\right)\right\|_1 \lesssim \lambda^{-\epsilon}\left\|f_1\right\|_{\infty}\left\|f_2\right\|_{\infty},\]
	whenever $\operatorname{supp}\hat{f}_1\subset\{\lambda \leq |\xi_1|\leq 2 \lambda\}$ and $\operatorname{supp} \hat{f}_2 \subset\left\{\left|\xi_2\right| \leq 2 \lambda\right\}.$
	
	\subsection{Spatial decomposition of the functions}
	We now introduce a structural decomposition of the function based on the frequency parameter $\lambda$. We redefine $\eta_m(x, y)=\eta\left(\lambda^\gamma(x, y)-m\right)$ for some $\gamma\in\left(\frac{1}{2}, 1\right)$ to be chosen later. Also, we denote $\widetilde{\eta}_m$ as the smooth function taking value one on the support of $\eta_m$ and supported in the cube $Q_{m}$ of side length $2\lambda^{-\gamma}$ centered at the point $\lambda^{-\gamma}m\in\R^2$. For $l=1,2$, we write
	\[
	\begin{aligned}
		f_l & =\sum_{m \in \Z^2} \widetilde{\eta}_m \eta_m f_l \\
		& =\sum_{m \in \Z^2} \widetilde{\eta}_m \widetilde{\psi}_\lambda *_l\left(\eta_m f_l\right)+\sum_{m \in \Z^2}\widetilde{\eta}_m\left(\eta_m\left(\widetilde{\psi}_\lambda *_l f_l\right)-\widetilde{\psi}_\lambda *_l\left(\eta_m f_l\right)\right).
	\end{aligned}
	\]
	
	We set $f_{l,m}=\widetilde{\psi}_\lambda *_l\left(\eta_m f_l\right)$ and $f_{l, m, err}= \eta_m\left(\widetilde{\psi}_\lambda *_l f_l\right)-\widetilde{\psi}_\lambda *_l\left(\eta_m f_l\right)$.
	As in \Cref{sec:spacelocal}, we have that $\left|f_{l, m, err}\right| \lesssim \lambda^{\gamma-1} \|f_l\|_{\infty}$ by mean value theorem. We write our operator as
	\[
	\begin{aligned}
		\left\|T_{loc}\left(f_1, f_2\right)\right\|_1 &\leq\left\|T_{loc}\left(\sum_m \widetilde{\eta}_m f_{1, m}, \sum_m \widetilde{\eta}_{m} f_{2, m}\right)\right\|_1 \\
		& \quad+\left\|T_{loc}\left(\sum_m \widetilde{\eta}_m f_{1, m}, \sum_m \widetilde{\eta}_m f_{2, m, err}\right)\right\|_1 \\
		& \quad+\left\|T_{loc}\left(\sum_m \widetilde{\eta}_m f_{1, m, err}, \sum_m \widetilde{\eta}_m f_{2, m}\right)\right\|_1 \\
		& \quad+\left\|T_{loc}\left(\sum_m \widetilde{\eta}_m f_{1, m, err}, \sum_m \widetilde{\eta}_m f_{2, m, err}\right)\right\|_1.
	\end{aligned}
	\]
	
	The last three terms in the expression above are bounded by a constant multiple of $\lambda^{\gamma-1}\left\|f_1\right\|_{\infty}\left\|f_2\right\|_{\infty}$. Therefore, it remains to estimate the first term.
	\subsection{Decomposition of the operator based on the degeneracy at $t=0$}\label{sec:degeneracy} Let $0<\tau<\gamma-\frac{1}{2}$ be a constant to be determined later. We decompose the operator as
	\[
	\begin{aligned}
		& T_{loc}(f_1,f_2)=\sum_{\ell=0}^{\left\lceil\log_2\left(\frac{\pi\lambda^\tau}{6}\right)\right\rceil} T_{loc}^\ell(f_1,f_2) \text {, where } \\
		& T_{loc}^0(f_1,f_2)(x, y)=\int_0^{\lambda^{-\tau}} f_1(x+\cos t, y) g(x, y+\sin t) \zeta(x,y,t) d t~ \text {and } \\
		& T_{loc}^\ell(f_1,f_2)(x, y)=\int_{2^{\ell} \lambda^{-\tau}}^{2^{\ell+1} \lambda^{-\tau}} f_1(x+\cos t, y) f_2(x, y+\sin t) \zeta(x,y,t)\; d t
	\end{aligned}
	\]
	for $\ell=1,\dots, \left\lceil\log_2\left(\frac{\pi\lambda^\tau}{6}\right)\right\rceil.$	We note that $\left\|T_{loc}^0\left(f_1, f_2\right)\right\|_1 \lesssim \lambda^{-\tau}\left\|f_1\right\|_\infty\left\|f_2\right\|_{\infty}$. Therefore the problem boils down to establishing $(\infty,\infty,1)-$bound for the operator:
	\[
	T_{loc}^\delta\left(f_1, f_2\right)(x, y)=\int_\delta^{2 \delta} f_1(x+\cos t, y) f_2(x, y+\sin t) \zeta(x,y,t) d t \text {, for } \lambda^{-\tau}\leq\delta<1.
	\]
	More precisely, we will establish the following estimate for $0<\delta<1$.
	\begin{equation}\label{desiredbound}
		\|T_{loc}^\delta(f_1,f_2)\|_1\lesssim\delta^{-1}\lambda^{-\mathfrak{c}}\|f_1\|_\infty\|f_2\|_\infty, \text{ for some constant }\mathfrak{c}>0.
	\end{equation}
	Indeed, suppose the inequality \eqref{desiredbound} is true for all $0<\delta<1$, then we have that
	\begin{align*}
		\|T_{loc}^\delta(f_1,f_2)\|_1\lesssim\sum\limits_{\ell=0}^{\left\lceil\log_2\left(\frac{\pi\lambda^\tau}{6}\right)\right\rceil}(2^\ell\lambda^{-\tau})^{-1}\lambda^{-\mathfrak{c}}\|f_1\|_\infty\|f_2\|_\infty\lesssim\lambda^{-\mathfrak{c}+\tau}\|f_1\|_\infty\|f_2\|_\infty.
	\end{align*} 
	Note that the desired estimate follows by choosing $\tau<\mathfrak{c}$. Therefore, we need to prove inequality \eqref{desiredbound}. In order to prove this inequality, we require a structural decomposition of the functions based on a frequency pruning lemma obtained in \cite{CDR}, which is the content of the next section.
	\subsection{Frequency Pruning and a structural decomposition of the functions.}\label{sec:pruning} For $s\in\R$, we define the multiplicative derivative by
	\[
	\begin{aligned}
		\mathscr{D}_s f(x)&=f(x+s) \overline{f(x)} \text {, for } f:\R\to\R,\\
		\mathscr{D}^{(1)}_s f(x,y)&=f(x+s,y) \overline{f(x,y)} \text {, for } f:\R^2\to\R,\\
		\mathscr{D}^{(2)}_s f(x,y)&=f(x,y+s) \overline{f(x,y)} \text {, for } f:\R^2\to\R.
	\end{aligned}
	\]
	
	We decompose the functions into two pieces, one having a small $L^2$-norm of the multiplicative derivative and the other having a controlled Fourier support for each fiber. For this purpose we state the frequency pruning lemma from \cite{CDR},
	\begin{lemma}[\cite{CDR} Lemma 3.2]\label{multiplicativelemma}
		Let $f \in L^2(\mathbb{R})$, $R \geq 1$ and $\varrho \in(0,1)$. There exists a decomposition
		\[
		f=f_{\sharp}+f_\flat,
		\]
		such that the following conditions hold true:
		\begin{enumerate}[(i)]
			\item We have the norm control,
			\[
			\left\|f_{\sharp}\right\|_{2}+\left\|f_\flat\right\|_{2} \lesssim\|f\|_{L^2}.
			\]
			\item The function $f_{\sharp}$ is given by
			\[f_{\sharp}(x)=\sum_{n=1}^{\mathcal{N}_\varrho} h_n(x) e^{i \alpha_n x}\]
			such that each $\alpha_n \in \mathbb{R}$ and $h_n$ is a smooth function satisfying
			\begin{enumerate}[(a)]
				\item $\left\|\partial^N h_n\right\|_{\infty} \lesssim_N R^N\|f\|_{\infty}$ for all integers $N \geq 0$.
				\item $\left\|h_n\right\|_{2} \lesssim\|f\|_{2}$.
				\item $\widehat h_n$ is supported in $[-R, R]$. Moreover, the support of $\widehat{f}_{\sharp}$ is contained in the support of $\widehat{f}$.
				\item $\mathcal{N}_\varrho\lesssim \varrho^{-1}$.
			\end{enumerate}
			\item One has the bound
			\[
			\int_{\mathbb{R}} \int_{|\xi| \leq R}\left|\widehat{\mathscr{D}_s f_\flat}(\xi)\right|^2 d \xi d s \lesssim \varrho\|f\|_{2}^4 .
			\]	
		\end{enumerate}
	\end{lemma}
	
	We apply the lemma as above to the maps $x \mapsto f_{1, m}(x, y)$ and $y \mapsto f_{2, m}(x, y)$ with $R=\delta\lambda^{\gamma+\mu_1}$ and $\varrho=\lambda^{-\mu_2}$ for some $\mu_1,\mu_2>0$ to obtain,
	\[
	\begin{aligned}
		f_{l, m} & =f_{l, m, \sharp}+f_{l, m, \flat}, \quad l=1,2 . \\
		f_{1, m, \sharp}(x, y) & =\sum_{n=1}^{\mathcal N_1} h_{1, m, n}(x, y) e^{i \alpha_{m, n}(y) x}, \\
		f_{2, m, \sharp}(x, y) & =\sum_{n=1}^{\mathcal N_2} h_{2, m, n}(x, y) e^{i \beta_{m, n}(x) y},
	\end{aligned}
	\]
	where $\left|\mathcal{N}_1\right|+\left|\mathcal{N}_2\right| \lesssim \lambda^{\mu_2}$, and $\alpha_{m,n}$ and $\beta_{m,n}$ are measurable functions with $\left|\alpha_{m, n}\right| \sim \lambda, \quad\left|\beta_{m, n}\right| \lesssim \lambda$.  Moreover, $h_{l, m, n}$ are smooth functions with $\left|\partial_l^N h_{l, m, n}\right| \lesssim (\delta\lambda^{\gamma+\mu_1})^N\left\|f_{l,m}\right\|_{\infty}$ uniformly in $m$ ad $n$. We have following control over the $L^2$-norm of the multiplicative derivative of $f_{l, m, \flat}$,
	\[
	\int_{\mathbb{R}}\int_{\left|\xi_2\right| \leq R}\left|\widehat{\mathscr{D}_s^{(l)} f_{l, m, \flat}}(\xi)\right|^2 d \xi d s \lesssim \lambda^{-\mu_2-3 \gamma}\left\|f_{l,m}\right\|_{\infty}^4.
	\]
	To see this, using \Cref{multiplicativelemma}(3) for $l=1$, we have
	\begin{align*}
		\int_{\mathbb{R}}\int_{\left|\xi_1\right| \leq R}\left|\widehat{\mathscr{D}_s^{(1)} f_{1, m, \flat}}(\xi)\right|^2 d \xi d s&\lesssim\lambda^{-\mu_2}\int\|\mathcal{F}_yf_{1,m,\flat}\|_{L^2_x}^4dy\\
		&\leq\lambda^{-\mu_2}\Big\|\|\mathcal{F}_yf_{1,m,\flat}\|_{L^4_y}\Big\|_{L^2_x}^4\\
		&=\lambda^{-\mu_2}\Big\|\|\mathcal{F}_yf_{1,m,\flat}\cdot\mathcal{F}_yf_{1,m,\flat}\|_{L^2_y}^{\frac{1}{2}}\Big\|_{L^2_x}^4\\
		&=\lambda^{-\mu_2}\Big\|\|\mathcal{F}_y(f_{1,m,\flat}\ast_{(2)} f_{1,m,\flat})\|_{L^2_y}^{\frac{1}{2}}\Big\|_{L^2_x}^4\\
		&=\lambda^{-\mu_2}\Big\|\|f_{1,m,\flat}\ast_{(2)} f_{1,m,\flat}\|_{L^2_y}^{\frac{1}{2}}\Big\|_{L^2_x}^4\\
		&\lesssim\lambda^{-\mu_2-3 \gamma}\left\|f_{1,m}\right\|_{\infty}^4,
	\end{align*}
	where we have used Plancherel's theorem in the second to last equality and support condition of $f_{1,m}$ in the last inequality. Similarly, we can obtain the estimate for $l=2$.
	
	\subsection{Conclusion of the proof of \Cref{Trilinearsmoothing}}Using the above frequency decomposition, we write,
	\[
	\begin{aligned}
		T_{loc}^{\delta}\Big(\sum_m \tilde{\eta}_m f_{1, m}, \sum_m \tilde{\eta}_m f_{2, m}\Big)=& T_{loc}^{\delta}\left(f_{1, \flat}, f_{2, \flat}\right)+T_{loc}^{\delta}\left(f_{1, \flat}, f_{2, \sharp}\right) +T_{loc}^\delta\left(f_{1, \sharp}, f_{2, \flat}\right)\\
		&+T_{loc}^{\delta}\left(f_{1, \sharp}, f_{2, \sharp}\right),
	\end{aligned}
	\]
	where $f_{l,b}=\sum\tilde{\eta}_mf_{l,m,b}$, and $f_{l,\sharp}=\sum\tilde{\eta}_mf_{l,m,\sharp},\;l=1,2$.
	
	The estimate for the first three terms involving $f_{l,\flat},\;l=1,2$ requires a delicate analysis based on oscillatory phase considerations. The decay estimates for those terms follow at once from the following lemma. We prove the lemma in \Cref{sec:Oscillatory}.
	\begin{lemma}\label{lemma:oscillatory}
		Let $\lambda,\delta,\mu_1,\mu_2>0$ and $\frac{1}{2}<\gamma<1$. Consider the sequences of functions $\{f_{l,m_l},\;m_l\in\Z^2\},\;l=1,2$ such that
		\begin{enumerate}[(i)]
			\item $f_{l,m_l}$ is supported in a cube $Q_{m_l}\subset\R^2$ of side length $4\lambda^{-\gamma}$ and centered at $\lambda^{-\gamma}m_l$ for $l=1,2$.
			\item $\|\partial_l^\alpha f_{l,m_l}\|_\infty\lesssim\lambda^{|\alpha|}\|f_{l,m_l}\|_\infty$, for $l=1,2,$ and all $\alpha\in\N\cup\{0\}.$
			\item For $l=1$ or $l=2$, we have
			\begin{equation}\label{smallMD}
				\int\limits_{\mathbb{R}}\int\limits_{\substack{\xi\in\R^2:\\\left|\xi_l\right| \leq \delta\lambda^{\gamma+\mu_1}}}\left|\widehat{\mathscr{D}_s^{(l)} f_{l, m_l}}(\xi)\right|^2 d \xi d s \lesssim \lambda^{-\mu_2-3\gamma}\left\|f_{l,m_l}\right\|_{\infty}^4.
			\end{equation}
		\end{enumerate}
		Then there exists constant $\mathfrak{c}_1>0$ (depending on $\mu_1,\mu_2,\gamma$) such that the following holds true,
		\[
		\left\|T^\delta_{loc}\left(\sum_{m_1}f_{1,m_1},\sum_{m_2}f_{2,m_2}\right)\right\|_1\lesssim \delta^{-1}\lambda^{-\mathfrak{c}_1}\sup_{m_1\in\Z^2}\|f_{1,m_1}\|_\infty \sup_{m_2\in\Z^2}\|f_{2,m_2}\|_\infty.
		\]
	\end{lemma}
	The $L^1-$bound for the term $T_{loc}^{\delta}\left(f_{1, \sharp}, f_{2, \sharp}\right)$ are obtained by a reduction to the analysis of sublevel sets $\mathcal{E}(\alpha,\beta,\epsilon)$ defined by
	\[\mathcal{E}(\alpha,\beta,\epsilon)=\{(x,y,t)\in K:\;|\alpha(x+\cos t,y)\sin t+\beta(x,y+\sin t)\cos t|\lesssim\epsilon\},\]
	for some $\epsilon>0$ and $\alpha,\beta$ measurable functions on $\R^2$.
	This reduction to the sublevel sets is presented in the following lemma, the proof of which is postponed to \Cref{sec:reductiontosublevel}.
	\begin{lemma}\label{lemma:reductiontosublevel}
		Let $K\subseteq\R^2\times I$ be a compact set where $I=[\delta,2\delta]$ for some $\delta>0$. Then for $\rho>0$, there exist real-valued measurable functions $\alpha_{n_1},\;n_1=1,\dots,\mathcal{N}_1$, and $\beta_{n_2},\;n_2=1,\dots,\mathcal{N}_2$ such that $|\alpha_{n_1}(t)|\sim1$ and $|\beta_{n_2}(t)|\lesssim1$ for all $t\in\R$ and the following holds,
		\begin{align*}
			\|T^\delta_{loc}(f_{1,\#},f_{2,\#})\|_1\lesssim& \delta\lambda^{-100}\|f_1\|_\infty\|f_2\|_\infty\\&+\|f_1\|_\infty\|f_2\|_\infty\sum_{n_1=1}^{\mathcal{N}_1}\sum_{n_2=1}^{\mathcal{N}_2}\mathcal{E}(\alpha_{n_1},\beta_{n_2},2\delta\lambda^{\gamma+\mu_1+\rho-1}).
		\end{align*}
	\end{lemma}
	Using the above lemma along with \Cref{Sublevelsetestimate}, we have
	\begin{align*}
		\|T^\delta_{loc}(f_{1,\#},f_{2,\#})\|_1&\lesssim \delta\lambda^{-100}\|f_1\|_\infty\|f_2\|_\infty+\|f_1\|_\infty\|f_2\|_\infty\sum_{n_1=1}^{\mathcal{N}_1}\sum_{n_2=1}^{\mathcal{N}_2}\delta\lambda^{\frac{\gamma+\mu_1+\rho-1}{35}}\\
		&\lesssim\delta(\lambda^{-100}+\lambda^{\frac{\gamma+\mu_1+\rho-1}{35}+2\mu_2})\|f_1\|_\infty\|f_2\|_\infty.
	\end{align*}
	We require $\mu_1,\mu_2,\rho$ to satisfy ${1-\gamma-\mu_1-\rho}-70\mu_2>0$. This concludes the proof of \Cref{Trilinearsmoothing}.
	\subsection{Proof of \Cref{lemma:oscillatory}}\label{sec:Oscillatory}
	Without loss of generality, we may assume that $\sup_{m_l\in\Z^2}\|f_{l,m_l}\|_\infty=1,\;l=1,2$. Moreover, we will prove the estimate when \eqref{smallMD} is true for $l=1$, the other case follows by a similar reasoning.
	
	Let $\mathcal M=\left\{\left(m_1, m_2\right) \in\left(\Z^2\right)^2:\left\|T\left(f_{1,m_1}, f_{2, m_2}\right)\right\|_1 \neq 0\right\}$ and\\ $\mathcal M_l=\left\{m_l \in \mathbb{Z}^2: \exists\left(m_1, m_2\right) \in \mathcal{M}\right\}$.
	
	We claim that $\# \mathcal M_l \lesssim \lambda^{2 \gamma},\;l=1,2$ and $\# \mathcal M \lesssim \lambda^{3 \gamma}$. Note that, due to the compact support of $\zeta$, it is easy to verify that $\# \mathcal M_l \lesssim \lambda^{2 \gamma},\;l=1,2$ and $\# \mathcal M \lesssim \lambda^{4 \gamma}$. Additionally, we have $(x+\cos t,y)\in Q_{m_1}$, $(x,y+\sin t)\in Q_{m_2}$ and
	\[[(y+\sin t)-y]^2=\sin^2t=1-\cos^2t=1-[(x+\cos t)-x]^2.\]
	The above conditions imply that
	\begin{eqnarray*}
		&&\lambda^{-2\gamma}(m_{2,2}-m_{1,2})^2+O(\lambda^{-2\gamma})=1-\lambda^{-2\gamma}(m_{1,1}-m_{2,1})^2\\
		&\implies& (m_{1,1}-m_{2,1})^2+(m_{2,2}-m_{1,2})^2=O(\lambda^{2\gamma}).
	\end{eqnarray*}
	Thus, for a given $m_1$, we have $O(\lambda^{\gamma})$ many choices of $m_2$ such that $(m_1,m_2)\in\mathcal{M}$ which in turn implies that $\# \mathcal M \lesssim \lambda^{3 \gamma}$.
	
	For each ${m}=\left(m_1, m_2\right) \in \mathcal{M}$, we fix $(\bar{x}, \bar{y}, \bar{t})=\left(\bar{x}_{{m}}, \bar{y}_{{m}}, \bar{t}_{{m}}\right)$ such that,
	\[
	(\bar{x}+\cos \bar{t}, \bar{y}) \in Q_{m_1} \quad \text{and} \quad(\bar{x}, \bar{y}+\sin \bar{t}) \in Q_{m_2}.
	\]
	Then, for $(x, y, t)$ with $(x+\cos t, y)\in Q_{m_1}$ and $(x, y+\sin t)\in Q_{m_2}$ we have that

	\[|x-\bar{x}|+|y-\bar{y}| \lesssim \lambda^{-\gamma}, \] and
	\begin{align*}|t-\bar{t}|&=\left|\sin t^{\prime}\right|^{-1}|\cos t-\cos \bar{t}|,~~ \exists t^{\prime} \in(t, \bar{t})\\
		&\lesssim \delta^{-1}\left(|x-\bar{x}|+|x+\cos t-(\bar{x}+\cos \bar{t})|\right) \lesssim\delta^{-1} \lambda^{-\gamma}.
	\end{align*}
	Now, $T\left(f_{1, m_1}, f_{2, m_2}\right)(x, y)=\int_{\delta}^{2 \delta} f_{1, m_1}(x+\cos t, y) f_{2,m_2}(x, y+\sin t) \zeta_m(x, y, t)\;d t$, where
	\[\zeta_m(x, y, t)=\zeta(x, y, t) \tilde{\eta}\left(\lambda^\gamma(x+\cos t, y)-m_1,\right) \tilde{\eta}\left(\lambda^\gamma(x, y+\sin t)-m_2\right).\]
	
	We note that $\left\|\partial^\alpha \zeta_m\right\|_{\infty}\lesssim \lambda^{\gamma|\alpha|} \quad \forall \quad \alpha=\left(\alpha_1, \alpha_2, \alpha_3\right) \in(\mathbb{N} \cup\{0\})^3.$
	
	By Cauchy- Schwartz inequality and a change of variable $s=s+t$, we obtain,
	\[
	\begin{aligned}
		&\|T(f_{1,m_1},f_{2,m_2})\|_1\\&\lesssim\lambda^{-\gamma}\bigg|\int_{\R^4}f_{1,m_1}(x+\cos t,y)f_{2,m_2}(x,y+\sin t)\\
		&\hspace{1.4cm}\times\overline{f_{1,m_1}(x+\cos(t+s),y)f_{2,m_2}(x,y+\sin(t+s))}\tilde{\zeta}_m(x,y,t,s)\;dxdydtds\bigg|^\frac{1}{2},
	\end{aligned}
	\]
	where $\tilde{\zeta}_m \in C_c^{\infty}\left(\mathbb{R}^2 \times(\delta, 4 \delta)^2\right)$ is supported in a rectangle in $\mathbb{R}^4$ with side lengths $\lambda^{-\gamma}\times\lambda^{-\gamma}\times\delta^{-1}\lambda^{-\gamma}\times\delta^{-1}\lambda^{-\gamma}$ and $\left\|\partial^\alpha \tilde{\zeta}_m\right\|_\infty\lesssim\lambda^{\gamma|\alpha|},\; \forall \alpha=\left(\alpha_1, \alpha_2, \alpha_3, \alpha_4\right)$. Moreover,
	\[
	\begin{aligned}
		|s| & =|t+s-t|=|\sin (t+s_0)|^{-1}| \cos (t+s)-\cos t|,\quad s_0\in (0, s)\\
		& \lesssim \delta^{-1}|x+\cos (t+s)-(x+\cos t)| \lesssim \delta^{-1} \lambda^{-\gamma}.
	\end{aligned}
	\]
	
	We linearize the phase functions in $s$ by writing,
	\[
	\begin{aligned}
		\cos (t+s)&=\cos t-(\sin t) s+O\left(\lambda^{-2\gamma}\right)\\
		& =\cos t-(\sin \bar t) s+(\sin \bar{t}-\sin t)s+O(\lambda^{-2\gamma}) \\
		& =\cos t-(\sin \bar t) s+O\left(\lambda^{-2 \gamma}\right) .
	\end{aligned}
	\]
	
	Similarly, $\sin (t+s)=\sin t+(\cos \bar{t})s+O\left(\lambda^{-2\gamma}\right)$.
	This implies that
	\[
	\begin{aligned}
		f_{1, m_1}(x+\cos (t+s), y)&=f_{1, m_1}(x+\cos t-\sin \bar{t} s, y)+O\left(\lambda^{1-2 \gamma}\right), \\
		f_{2, m_2}(x, y+\sin (t+s))&=f_{2, m_2}(x, y+\sin t+\cos \bar{t} s)+O\left(\lambda^{1- 2\gamma}\right).
	\end{aligned}
	\]
	Therefore, we have
	\begin{align*}
		\|T(f_{1,m_1},f_{2,m_2})\|_{1}
		\lesssim&\lambda^{-\gamma}\bigg|\int_{\R^4}\mathscr{D}^{(1)}_{-(\sin\bar{t})s}f_{1,m_1}(x+\cos t,y)\mathscr{D}^{(2)}_{(\cos\bar{t})s}f_{2,m_2}(x,y+\sin t)\\
		&\hspace{1.5 cm}\times\tilde{\zeta}_m(x,y,t,s)\;dxdydtds\bigg|^{\frac{1}{2}}\\
		&+\lambda^{-\gamma}\bigg|\int_{|s|\lesssim\delta^{-1}\lambda^{-\gamma}}\int_{\R^3}\lambda^{1-2\gamma}\tilde{\zeta}_m(x,y,t,s)\;dxdydtds\bigg|^{\frac{1}{2}}\\
		&+\lambda^{-\gamma}\bigg|\int_{|s|\lesssim\delta^{-1}\lambda^{-\gamma}}\int_{\R^3}\lambda^{2-4\gamma}\tilde{\zeta}_m(x,y,t,s)\;dxdydtds\bigg|^{\frac{1}{2}}
	\end{align*}
	The last two terms can be dominated by $\delta^{-1}\lambda^{1-5\gamma}$. By Cauchy-Schwarz inequality in ${m}$, we get
	\begin{equation}\label{intereqn1}
		\begin{aligned}
			&\|T(\sum_{m_1\in\Z^2}f_{1,m_1},\sum_{m_2\in\Z^2}f_{2,m_2})\|_{1}\\
			\lesssim&\lambda^{-\gamma}\Bigg[\sum_{{m}\in\mathcal{M}}\bigg|\int_{\R^4}\mathscr{D}^{(1)}_{-\sin\bar{t}s}f_{1,m_1}(x+\cos t,y)\mathscr{D}^{(2)}_{\cos\bar{t}s}f_{2,m_2}(x,y+\sin t)\\
			&\hspace{2.5cm}\times\tilde{\zeta}_m(x,y,t,s)\;dxdydtds\bigg|\Bigg]^{\frac{1}{2}}+\delta^{-1}\lambda^{1-2\gamma}
		\end{aligned}
	\end{equation}
	
	We use local Fourier series expansion of $\mathscr{D}^{(l)}_sf_l,\;l=1,2$ to get
	\[\mathscr{D}^{(l)}_sf_{l,m_l}(x,y)=\sum_{k\in\Z^2}a_{l,m_l,k,s}e^{\pi i\lambda^\gamma k\cdot(x,y)}\quad\text{ for }\quad x,y\in Q_{m_l},\]
	where $a_{l,m_l,k,s}=\frac{\lambda^{2\gamma}}{4}\int_{\R^2}\mathscr{D}^{(1)}_sf_{l,m_l}(x,y)e^{\pi i\lambda^\gamma k\cdot(x,y)}\;dxdy$. We have the following uniform bounds in $s$ for the coefficients $a_{l,m_l,k,s},\;l=1,2$.
	\begin{align}
		|a_{l,m_l,k,s}|&\lesssim\frac{\lambda^{2\gamma+(1-\gamma)(N_1+N_2)}}{|k_1|^{N_1}|k_2|^{N_2}},\quad\forall N_1,N_2\in\N\cup\{0\},\label{est:coefficient},\\
		\sum_{k_l\in\Z^2}|a_{l,m_l,k_l,s}|^2&\lesssim\lambda^{2\gamma}\|\mathscr{D}^{(l)}_sf_{l,m_l}\|_2^2\lesssim1.\label{est:l2coefficient}
	\end{align}
	The first one is a direct application of integration by parts and the other follows by the support of the functions $f_{l,m_l}$.
	We substitute the above Fourier expansion in \Cref{intereqn1} so that the first term can be rewritten as
	\begin{align*}
		&\lambda^\frac{\gamma}{2}\Bigg[\sum_{{m}\in\mathcal{M}}\Bigg|\int_{|s|\lesssim\delta^{-1}\lambda^{-\gamma}}\sum_{(k_1,k_2)\in(\Z^2)^2}a_{1,m_1,k_1,-(\sin\bar{t})s}\;a_{2,m_2,k_2,(\cos\bar{t})s}\\
		&\hspace{3cm}\int_{\R^3}e^{\pi i\lambda^\gamma[k_1\cdot(x+\cos t,y)+k_2\cdot(x,y+\sin t)]}\tilde{\zeta}_m(x,y,t,s)\;dxdydtds\Bigg|\Bigg]^\frac{1}{2}\\
		&\lesssim\sum_{j=1}^3\lambda^\frac{\gamma}{2}\Bigg[\sum_{{m}\in\mathcal{M}}\Bigg|\int_{|s|\lesssim\delta^{-1}\lambda^{-\gamma}}\sum_{(k_1,k_2)\in L_j}a_{1,m_1,k_1,-(\sin\bar{t})s}\;a_{2,m_2,k_2,(\cos\bar{t})s}\\
		&\hspace{3cm}\int_{\R^3}e^{\pi i\lambda^\gamma[k_1\cdot(x+\cos t,y)+k_2\cdot(x,y+\sin t)]}\tilde{\zeta}_m(x,y,t,s)\;dxdydtds\Bigg|\Bigg]^\frac{1}{2}\\
		&=:\sum_{j=1}^3\mathfrak{I}_j,
	\end{align*}
	where for $\epsilon_1,\epsilon_2>0$ (to be chosen later), we decompose $(\Z^2)^2=L_1\cup L_2\cup L_3$ with
	\begin{align*}
		L_1&=\{(k_1,k_2)\in(\Z^2)^2 : |k_{1,1}|\geq\lambda^{1-\gamma+\epsilon_1}\},\\
		L_2&=\{(k_1,k_2)\in(\Z^2)^2 : |k_{1,1}|<\lambda^{1-\gamma+\epsilon_1}\text{ and }|k_{2,2}|\geq\lambda^{1-\gamma+\epsilon_2}\},\\
		L_3&=\{(k_1,k_2)\in(\Z^2)^2 : |k_{1,1}|<\lambda^{1-\gamma+\epsilon_1}\text{ and }|k_{2,2}|<\lambda^{1-\gamma+\epsilon_2}\}.
	\end{align*}	
	To estimate $\mathfrak{I}_1$, we use \Cref{est:coefficient} and the support of $\tilde\zeta_m$ to obtain,
	\begin{align*}
		\mathfrak{I}_1&\lesssim\lambda^\frac{\gamma}{2}\Bigg[\sum_{{m}\in\mathcal{M}}\Bigg|\int_{|s|\lesssim\delta^{-1}\lambda^{-\gamma}}\sum_{(k_1,k_2)\in L_1}\lambda^{4\gamma+(1-\gamma)(N+6)}\frac{1}{|k_{1,1}|^N}\frac{1}{(1+|k_{1,2}|)^2}\\
		&\hspace{5.2cm}\times\frac{1}{(1+|k_{2,1}|)^2}\frac{1}{(1+|k_{2,2}|)^2}\delta^{-1}\lambda^{-3\gamma}\Bigg|\Bigg]^\frac{1}{2}\\
		&\lesssim\delta^{-1}\lambda^{3-\gamma-\frac{N\epsilon_1}{2}}\lesssim\delta^{-1}\lambda^{-100},
	\end{align*}
	where we used $\# \mathcal M \lesssim \lambda^{3 \gamma}$ in the last step and we choose $N\in\N$ such that $-3+\gamma+\frac{N\epsilon_1}{2}>100$ to obtain the required decay in $\lambda$.
	The estimate of $\mathfrak{I}_2$ follows by similar arguments.
	
	To estimate the term $\mathfrak{I}_3$, we observe that the gradient of the phase function,
	\[\mathcal{P}_\lambda(x,y,t)=\lambda^\gamma[(k_{1,1}+k_{2,1})x+(k_{1,2}+k_{2,2}y+k_{1,1}\cos t+k_{2,2}\sin t)],\]
	is given by
	\[\nabla \mathcal{P}_\lambda(x,y,t)=\lambda^\gamma[(k_{1,1},k_{1,2},0)+(k_{2,1},k_{2,2},0)+(0,0,-k_{1,1}\sin t+k_{2,2}\cos t)].\]
	In view of the phase estimate, we define $L_{3,1}=\{(k_1,k_2)\in L_3:\;|k_1+k_2|\leq\lambda^{\epsilon_3}\;\text{and}\;|-k_{1,1}\sin \bar t+k_{2,2}\cos \bar t|\leq\lambda^{\epsilon_3}\}$ for some $\epsilon_3>0$ to be chosen later. Hence, we have
	\begin{align*}
		\mathfrak{I}_3&\lesssim\lambda^\frac{\gamma}{2}\Bigg[\sum_{{m}\in\mathcal{M}}\Bigg|\int_{|s|\lesssim\delta^{-1}\lambda^{-\gamma}}\sum_{(k_1,k_2)\in L_{3,1}}a_{1,m_1,k_1,-(\sin\bar{t})s}\;a_{2,m_2,k_2,(\cos\bar{t})s}\\
		&\hspace{2.8cm}\int_{\R^3}e^{\pi i\lambda^\gamma[k_1\cdot(x+\cos t,y)+k_2\cdot(x,y+\sin t)]}\tilde{\zeta}_m(x,y,t,s)\;dxdydtds\Bigg|\Bigg]^\frac{1}{2}\\
		&+\lambda^\frac{\gamma}{2}\Bigg[\sum_{{m}\in\mathcal{M}}\Bigg|\int_{|s|\lesssim\delta^{-1}\lambda^{-\gamma}}\sum_{(k_1,k_2)\in L_3\setminus L_{3,1}}a_{1,m_1,k_1,-(\sin\bar{t})s}\;a_{2,m_2,k_2,(\cos\bar{t})s}\\
		&\hspace{2.8cm}\int_{\R^3}e^{\pi i\lambda^\gamma[k_1\cdot(x+\cos t,y)+k_2\cdot(x,y+\sin t)]}\tilde{\zeta}_m(x,y,t,s)\;dxdydtds\Bigg|\Bigg]^\frac{1}{2}\\
		&=:\mathfrak{I}_{3,1}+\mathfrak{I}_{3,2}.
	\end{align*}
	\subsection*{Estimate of $\mathfrak{I}_{3,2}$}We observe that for $(k_1,k_2)\in L_3\setminus L_{3,1}$, either $|k_1+k_2|\geq\lambda^{\epsilon_3}$ or $|-k_{1,1}\sin \bar t+k_{2,2}\cos \bar t|\geq\lambda^{\epsilon_3}$. In the latter case, we have
	\begin{equation}\label{observation1}
		\begin{aligned}
			&|-k_{1,1}\sin t+k_{2,2}\cos t|\\
			&\geq |-k_{1,1}\sin \bar t+k_{2,2}\cos \bar t|-|-k_{1,1}(\sin t-\sin \bar t)+k_{2,2}(\cos t -\cos \bar t)|\\
			&\geq|-k_{1,1}\sin \bar t+k_{2,2}\cos \bar t|-(|k_{1,1}|+|k_{2,2}|)|t-\bar t|\\
			&\geq|-k_{1,1}\sin \bar t+k_{2,2}\cos \bar t|-\lambda^{1-\gamma}(\lambda^{\epsilon_1}+\lambda^{\epsilon_2})\lambda^{-\gamma}\\
			&\geq|-k_{1,1}\sin \bar t+k_{2,2}\cos \bar t|-2\lambda^{1-2\gamma+\epsilon_3}\\
			&\gtrsim|-k_{1,1}\sin \bar t+k_{2,2}\cos \bar t|,
		\end{aligned}
	\end{equation}
	where we require that $\max\{\epsilon_1,\epsilon_2\}<\epsilon_3$ in the second to last step. Hence, by Cauchy-Schwarz inequality in $(k_1,k_2)$ and \eqref{est:l2coefficient}, we obtain
	\begin{align*}
		\mathfrak{I}_{3,2}&\leq\lambda^\frac{\gamma}{2}\Biggg[\sum_{{m}\in\mathcal{M}}\int_{|s|\lesssim\delta^{-1}\lambda^{-\gamma}}\left(\sum_{(k_1,k_2)\in(\Z^2)^2}|a_{1,m_1,k_1,-(\sin\bar{t})s}\;a_{2,m_2,k_2,(\cos\bar{t})s}|^2\right)^\frac{1}{2}\\
		&\hspace{3cm}\times\Bigg(\sum_{(k_1,k_2)\in L_3\setminus L_{3,1}}\Bigg|\int_{\R^3}e^{\pi i\lambda^\gamma[k_1\cdot(x+\cos t,y)+k_2\cdot(x,y+\sin t)]}\\
		&\hspace{6.2cm}\times\tilde{\zeta}_m(x,y,t,s)\;dxdydt\Bigg|^2\Bigg)^\frac{1}{2}ds\Biggg]^\frac{1}{2}\\
		&\lesssim\lambda^\frac{\gamma}{2}\Bigg[\sum_{(k_1,k_2)\in L_3\setminus L_{3,1}}\big|\lambda^{-\gamma N}\max\{|k_1+k_2|,|-k_{1,1}\sin\bar t+k_{2,2}\cos\bar t|\}^{-N}\\
		&\hspace{3.4cm}\times\lambda^{\gamma N}\delta^{-1}\lambda^{-3\gamma}\big|^2\Bigg]^\frac{1}{4}\Bigg[\sum_{{m}\in\mathcal{M}}\int_{|s|\lesssim\delta^{-1}\lambda^{-\gamma}}ds\Bigg]^\frac{1}{2}\\
		&\lesssim\lambda^\frac{\gamma}{2}\left(\lambda^{-\frac{N\epsilon_3}{2}-\frac{3\gamma}{2}}\delta^{-\frac{1}{2}}\right)\left(\delta^{-1}\lambda^{2\gamma}\right)^\frac{1}{2}\\
		&=\delta^{-1}\lambda^{-100},
	\end{align*}
	where we used integration by parts in the variables $x,y,$ and $t$ and \eqref{observation1} in the second step. We now determine $N\in\N$ such that $N\epsilon_3>200$.
	
	\subsection*{Estimate of $\mathfrak{I}_{3,1}$} To bound the term $\mathfrak{I}_{3,1}$, we will ignore the phase function unlike the estimate for $\mathfrak{I}_{3,2}$. Instead, we will rely on a counting argument and the small multiplicative derivative hypothesis \eqref{smallMD}. We begin by noting that for $(k_1,k_2)\in L_{3,1}$, we have
	\[|k_{1,1}\sin\bar t+k_{1,2}\cos\bar t|\leq|-k_{1,1}\sin\bar t+k_{2,2}\cos\bar t| +|\cos\bar t||k_{2,2}+k_{1,2}|\lesssim\lambda^{\epsilon_3}.\]
	We also observe that for a fixed $k_2$ we have $O(\lambda^{2\epsilon_3})$ many $k_1$'s and vice-versa. This coupled with Cauchy-Schwarz inequality implies that
	\begin{eqnarray}
		\mathfrak{I}_{3,1}&\lesssim&\lambda^\frac{\gamma}{2}\Bigg[\sum_{{m}\in\mathcal{M}}\int_{|s|\lesssim\delta^{-1}\lambda^{-\gamma}}\Bigg(\sum_{(k_1,k_2)\in L_{3,1}}|a_{1,m_1,k_1,-(\sin\bar{t})s}|^2\Bigg)^\frac{1}{2}\nonumber\\
		&&\hspace{4cm}\Bigg(\sum_{(k_1,k_2)\in L_{3,1}}|a_{2,m_2,k_2,(\cos\bar{t})s}|^2\Bigg)^\frac{1}{2}\lambda^{-3\gamma}\delta^{-1}\;ds\Bigg]^\frac{1}{2}\nonumber\\
		&\lesssim&\delta^{-\frac{1}{2}}\lambda^{-\gamma+\epsilon_3}\Biggg[\sum_{{m}\in\mathcal{M}}\int_{|s|\lesssim\delta^{-1}\lambda^{-\gamma}}\Bigg(\sum_{\substack{k_1\in\Z^2:\\|k_{1,1}\sin\bar t+k_{1,2}\cos\bar t|\lesssim\lambda^{\epsilon_3}}}|a_{1,m_1,k_1,-(\sin\bar{t})s}|^2\Bigg)^\frac{1}{2}\nonumber\\
		&&\hspace{5.2cm}\left(\sum_{k_2\in\Z^2}|a_{2,m_2,k_2,(\cos\bar{t})s}|^2\right)^\frac{1}{2}\;ds\Biggg]^\frac{1}{2}\nonumber\\
		&\lesssim&\delta^{-\frac{1}{2}}\lambda^{-\gamma+\epsilon_3}\left[\sum_{{m}\in\mathcal{M}}\int_{|s|\lesssim\delta^{-1}\lambda^{-\gamma}}\left(\sum_{\substack{k_1\in\Z^2:\\|k_{1,1}\sin\bar t+k_{1,2}\cos\bar t|\lesssim\lambda^{\epsilon_3}}}|a_{1,m_1,k_1,-(\sin\bar{t})s}|^2\right)^\frac{1}{2}\;ds\right]^\frac{1}{2}\nonumber\\
		&\lesssim&\delta^{-\frac{1}{2}}\lambda^{-\gamma+\epsilon_3}\left[\sum_{{m}\in\mathcal{M}}\int_{|s|\lesssim\lambda^{-\gamma}}\left(\sum_{\substack{k_1\in\Z^2:\\|k_{1,1}\sin\bar t+k_{1,2}\cos\bar t|\lesssim\lambda^{\epsilon_3}}}|a_{1,m_1,k_1,s}|^2\right)^\frac{1}{2}\;\delta^{-1}ds\right]^\frac{1}{2}\nonumber\\
		&\lesssim&\delta^{-1}\lambda^{-\frac{\gamma}{2}+\epsilon_3}\left[\sum_{{m}\in\mathcal{M}}\int_{|s|\lesssim\lambda^{-\gamma}}\sum_{\substack{k_1\in\Z^2:\\|k_{1,1}\sin\bar t+k_{1,2}\cos\bar t|\lesssim\lambda^{\epsilon_3}}}|a_{1,m_1,k_1,s}|^2\;ds\right]^\frac{1}{4}\label{eqn:intermediate1},
	\end{eqnarray}
	where we used \eqref{est:l2coefficient} in the second step and two applications of Cauchy-Schwarz inequality along with $\#\mathcal{M}\lesssim\lambda^{3\gamma}$ in the last step.
	
	\subsection*{Claim}
	Next, we claim that for fixed $m_1,k_1\in\Z^2$, there are atmost $O(1+\delta\lambda^{\gamma+\epsilon_3}|k_{1,1}|^{-1})$ many choices of $m_2\in\Z^2$ such that $|k_{1,1}\sin\bar t+k_{1,2}\cos\bar t|\lesssim\lambda^{\epsilon_3}$.\\
	Let us first verify the claim. Let $m_2,m_2'\in\Z^2$ be two indices satisfying the hypothesis of the claim and let ${m}=(m_1,m_2)$ and ${m'}=(m_1',m_2')$. Then, we have
	\begin{align*}
		|\bar t_{{m}}-\bar t_{{m'}}|&\gtrsim\frac{|\cos(\bar t_{{m}})-\cos (\bar t_{{m'}})|}{|\sin t'|},\quad\text{for some}\quad t'\in(\bar t_{{m}},\bar t_{{m'}})\\
		&\gtrsim\frac{1}{\delta}\left[\left|\left(\bar x_{{m}}+\cos(\bar t_{{m}})\right)-\left(\bar x_{{m'}}+\cos (\bar t_{{m'}})\right)|-|\bar x_{{m}}-\bar x_{{m'}}\right|\right]\\
		&\geq \frac{\lambda^{-\gamma}}{\delta}\left[|m_{2,1}-m_{2,1}'|-4\right]\\
		&\gtrsim\frac{\lambda^{-\gamma}}{\delta}|m_{2,1}-m_{2,1}'|,\quad \text{for}\quad |m_{2,1}-m_{2,1}'|\geq 8.
	\end{align*}
	Hence the hypothesis and the above observation implies that,
	\begin{align*}
		\lambda^{\epsilon_3}&\gtrsim|k_{1,1}\sin\bar t_{{m}}+k_{1,2}\cos\bar  t_{{m}}|+|k_{1,1}\sin\bar t_{{m'}}+k_{1,2}\cos\bar t_{{m'}}|\\
		&\gtrsim|k_{1,1}\tan\bar t_{{m}}+k_{1,2}-k_{1,1}\tan\bar t_{{m'}}-k_{1,2}|\\
		&\geq|k_{1,1}||\sec^2 t'||\bar t_{{m}}-\bar t_{{m'}}|,\quad \text{for some}\quad t'\in(\bar t_{{m}},\bar t_{{m'}})\\
		&\gtrsim\lambda^{-\gamma}\delta^{-1}|k_{1,1}||m_{2,1}-m_{2,1}'|.
	\end{align*}
	Hence, the claimed bound follows as for fixed $m_1$ and $m_{2,1}$, there are only $O(1)$ choices of $m_{2,2}$.

	To exploit the decay in $k_{1,1}$ in the claim, we divide the term in \eqref{eqn:intermediate1} into two parts such that $|k_{1,1}|\geq\delta\lambda^{\mu_1}$ or $|k_{1,1}|<\delta\lambda^{\mu_1}$. Since $\#\{m_1\in\Z^2:{m}=(m_1,m_2)\in\mathcal{M}\}\leq\lambda^{2\gamma}$, we have that
	\begin{align*}
		\mathfrak{I}_{3,1}&\lesssim\delta^{-1}\lambda^{-\frac{\gamma}{2}+\epsilon_3}\|f_{2,m_2}\|_\infty^\frac{1}{2}\left(\int_{|s|\lesssim\lambda^{-\gamma}}\sum_{{m}\in\mathcal{M}}\sum_{\substack{k_1\in\Z^2:\\|k_{1,1}|<\delta\lambda^{\mu_1}}}|a_{1,m_1,k_1,s}|^2\;ds\right)^\frac{1}{4}\\
		&+\delta^{-1}\lambda^{-\frac{\gamma}{2}+\epsilon_3}\|f_{2,m_2}\|_\infty^\frac{1}{2}\left(\int_{|s|\lesssim\lambda^{-\gamma}}\sum_{{m}\in\mathcal{M}}\sum_{\substack{k_1\in\Z^2:\\|k_{1,1}\sin\bar t+k_{1,2}\cos\bar t|\lesssim\lambda^{\epsilon_3}\\|k_{1,1}|\geq\delta\lambda^{\mu_1}}}|a_{1,m_1,k_1,s}|^2\;ds\right)^\frac{1}{4}\\
		&=:\mathfrak{I}_{3,1,1}+\mathfrak{I}_{3,1,2}.
	\end{align*}
	To estimate $\mathfrak{I}_{3,1,1}$, we use the fact that for a fixed $m_1\in\mathcal{M}_1$ there are atmost $O(\lambda^{\gamma})$ many $m_2\in\mathcal{M}_2$ to obtain
	\begin{align*}
		\mathfrak{I}_{3,1,1}&\lesssim\delta^{-1}\lambda^{-\frac{\gamma}{2}+\epsilon_3+\frac{\gamma}{4}}\left(\int_{|s|\lesssim\lambda^{-\gamma}}\sum_{m_1\in\mathcal{M}_1}\sum_{\substack{k_1\in\Z^2:\\|k_{1,1}|<\delta\lambda^{\mu_1}}}|a_{1,m_1,k_1,s}|^2\;ds\right)^\frac{1}{4}\\
		&\lesssim\delta^{-1}\lambda^{-\frac{\gamma}{4}+\epsilon_3}\left(\lambda^{2\gamma}\sum_{m_1\in\mathcal{M}_1}\int_{s\in\R}\int_{|\xi_1|\lesssim\delta\lambda^{\gamma+\mu_1}}\left|\widehat{\mathscr{D}_s^{(1)} f_{1, m_1}}(\xi)\right|^2 d \xi\;ds\right)^\frac{1}{4}\\
		&\lesssim\delta^{-1}\lambda^{-\frac{\gamma}{4}+\epsilon_3}\left(\lambda^{2\gamma}\lambda^{2\gamma}\lambda^{-\mu_2-3\gamma}\right)^\frac{1}{4}\\
		&\lesssim\delta^{-1}\lambda^{-\frac{\mu_2}{4}+\epsilon_3},
	\end{align*}
	where we used the hypothesis \eqref{smallMD} in the second last step and restrict $\epsilon_3>0$ to satisfy $\mu_2>4\epsilon_3$.
	
	In order to prove required estimate for the term $\mathfrak{I}_{3,1,2}$, we use the above claim to obtain
	\begin{align*}
		\mathfrak{I}_{3,1,2}&\lesssim\delta^{-1}\lambda^{-\frac{\gamma}{2}+\epsilon_3}\left(\int_{|s|\lesssim\lambda^{-\gamma}}\sum_{m_1\in\mathcal{M}_1}\sum_{\substack{k_1\in\Z^2:\\|k_{1,1}|\geq\delta\lambda^{\mu_1}}}(1+\delta\lambda^{\gamma+\epsilon_3}|k_{1,1}|^{-1})|a_{1,m_1,k_1,s}|^2\;ds\right)^\frac{1}{4}\\
		&\lesssim\delta^{-1}\lambda^{-\frac{\gamma}{4}+\frac{5\epsilon_3}{4}-\frac{\mu_1}{4}}\left(\int_{|s|\lesssim\lambda^{-\gamma}}\sum_{m_1\in\mathcal{M}_1}\sum_{k_1\in\Z^2}|a_{1,m_1,k_1,s}|^2\;ds\right)^\frac{1}{4}\\
		&\lesssim\delta^{-1}\lambda^{\frac{5\epsilon_3}{4}-\frac{\mu_1}{4}},
	\end{align*}
	where we require that $5\epsilon_3<\mu_1<\gamma+\epsilon_3$.
	This concludes the proof of \Cref{lemma:oscillatory}.
	
	\subsection{Proof of \Cref{lemma:reductiontosublevel}}\label{sec:reductiontosublevel}
	We note that $\|T^\delta_{loc}(f_{1,\#},f_{2,\#})\|_1$ is dominated by $O(\lambda^{2\mu_2})$ many terms of the form,
	\begin{equation}\label{eqn:intermediate2}
		\sum_{{m}\in(\Z^2)^2}\iint\left|\int_\delta^{2\delta}e^{i(\alpha_{m_1,n_1}(y)\cos t+\beta_{m_2,n_2}(x)\sin t)}K_{{m}}(x,y,t)\;dt\right|dxdy,
	\end{equation}
	where $K_{{m}}(x,y,t)=(\tilde\eta_{m_1}h_{1,m_1,n_1})(x+\cos t,y)(\tilde\eta_{m_2}h_{2,m_2,n_2})(x,y+\sin t)\zeta(x,y,t)$ with $|\alpha_{m_1,n_1}|\sim\lambda$ and $|\beta_{m_2,n_2}|\lesssim\lambda$.
	
	We also observe that $\|\partial_t^N K_{{m}}\|_\infty\lesssim(\delta\lambda^{\gamma+\mu_1})^N\|f_1\|_\infty\|f_2\|_\infty$. We also recall that the functions $\tilde\eta_m$ are supported in cubes $Q_m$ which have finite bounded overlap. We will however, require the support to be disjoint and hence we decompose the cubes into $O(1)$ smaller cubes so that the resulting collection $\mathcal{Q}$ have pairwise disjoint cubes. Hence the quantity \eqref{eqn:intermediate2} is dominated by
	\begin{equation*}
		\sum_{\substack{{m}\in(\Z^2)^2:\\Q_{m_1}\in\mathcal{Q},\;Q_{m_2}\in\mathcal{Q}}}\iint\left|\int_\delta^{2\delta}e^{i(\alpha_{m_1,n_1}(y)\cos t+\beta_{m_2,n_2}(x)\sin t)}K_{{m}}(x,y,t)\;dt\right|dxdy,
	\end{equation*}
	Define $\widetilde{\mathcal{M}}$ to be the collection of all ${m}\in(\Z^2)^2$ such that $Q_{m_1}\in\mathcal{Q}$, $Q_{m_2}\in\mathcal{Q}$, and 
	\[\iint\left|\int_\delta^{2\delta}e^{i(\alpha_{m_1,n_1}(y)\cos t+\beta_{m_2,n_2}(x)\sin t)}K_{{m}}(x,y,t)\;dt\right|dxdy\neq0.\]
	Then, by a previous argument as in the proof of \Cref{lemma:oscillatory}, we have $\#\mathcal{\widetilde M}\lesssim\lambda^{3\gamma}$.
	
	For each ${m}\in\mathcal{\widetilde{M}}$, we fix $(\bar x,\bar y,\bar t)=(\bar x_{{m}},\bar y_{{m}},\bar t_{{m}})\in\operatorname*{supp}(K_{{m}})$. Since $|t-\bar t|\lesssim\delta^{-1}\lambda^{-\gamma}$, the phase function $\mathcal{\widetilde{P}}(x,y,t)=\alpha_{m_1,n_1}(y)\cos t+\beta_{m_2,n_2}(x)\sin t$ satisfies
	\begin{align*}
		|\partial_t\mathcal{\widetilde{P}}(x,y,t)|\geq& |-\alpha_{m_1,n_1}(y)\sin\bar t+\beta_{m_2,n_2}(x)\cos\bar t|-|\alpha_{m_1,n_1}(y)(\sin \bar t-\sin t)|\\
		&-|\beta_{m_2,n_2}(x)(\cos\bar t-\cos t)|\\
		\gtrsim&|-\alpha_{m_1,n_1}(y)\sin\bar t+\beta_{m_2,n_2}(x)\cos\bar t|-O(\delta^{-1}\lambda^{1-\gamma})\\
		\gtrsim&\delta\lambda^{\gamma+\mu_1+\rho},
	\end{align*}
	whenever $|-\alpha_{m_1,n_1}(y)\sin\bar t+\beta_{m_2,n_2}(x)\cos\bar t|\geq\delta\lambda^{\gamma+\mu_1+\rho}$, as $\tau<\gamma-\frac{1}{2}$.
	We write,
	\begin{align*}
		&\sum_{\substack{{m}\in\mathcal{\widetilde{M}}:\\Q_{m_1}\in\mathcal{Q},\;Q_{m_2}\in\mathcal{Q}}}\iint\bigg|\int_\delta^{2\delta}e^{i(\alpha_{m_1,n_1}(y)\cos t+\beta_{m_2,n_2}(x)\sin t)}K_{{m}}(x,y,t)\;dt\bigg|dxdy\\
		&\leq\sum_{\substack{{m}\in\mathcal{\widetilde{M}}:\\Q_{m_1}\in\mathcal{Q},\;Q_{m_2}\in\mathcal{Q}}}\iint_{\partial_t\mathcal{\widetilde{P}}(x,y,\bar t)\geq\delta\lambda^{\gamma+\mu_1+\rho}}\bigg|\int_\delta^{2\delta}e^{i(\alpha_{m_1,n_1}(y)\cos t+\beta_{m_2,n_2}(x)\sin t)}\\
		&\hspace{8cm}\times K_{{m}}(x,y,t)\;dt\bigg|dxdy\\
		&+\sum_{\substack{{m}\in\mathcal{\widetilde{M}}:\\Q_{m_1}\in\mathcal{Q},\;Q_{m_2}\in\mathcal{Q}}}\iint_{\partial_t\mathcal{\widetilde{P}}(x,y,\bar t)<\delta\lambda^{\gamma+\mu_1+\rho}}\bigg|\int_\delta^{2\delta}e^{i(\alpha_{m_1,n_1}(y)\cos t+\beta_{m_2,n_2}(x)\sin t)}\\
		&\hspace{8cm}\times K_{{m}}(x,y,t)\;dt\bigg|dxdy.
	\end{align*}
	Hence, applying integration by parts in the $t$ variable along with the estimate $\|\partial_t^N K_{{m}}\|_\infty\lesssim(\delta\lambda^{\gamma+\mu_1})^N\|f_1\|_\infty\|f_2\|_\infty$ implies that the first term is dominated by a constant multiple of $\delta\lambda^{-\rho N}\|f_1\|_\infty\|f_2\|_\infty$, for all $N\in\N$. 
	
	To bound the second term, we note that for each $n_1=1,\dots,\mathcal{N}_1,\;n_2=1,\dots,\mathcal{N}_2$ and $(x,y,t)\in\operatorname*{supp}\zeta$ there exists exactly one $m_1,m_2\in\Z^2$ such that $(x+\cos t,y)\in Q_{m_1}\in\mathcal{Q}$ and $(x,y+\sin t)\in Q_{m_2}\in\mathcal{Q}$. We define 
	\begin{align*}
		\alpha_{n_1}(x+\cos t,y)&=\begin{cases}-\lambda^{-1}\alpha_{m_1,n_1}(x+\cos t,y),\quad &\text{if}\;(x+\cos t,y)\in Q_{m_1}\\0,&\text{otherwise}\end{cases},\\
		\beta_{n_2}(x,y+\sin t)&=\begin{cases}\lambda^{-1}\beta_{m_2,n_2}(x,y+\sin t),\quad &\;\;\;\text{if}\;(x,y+\sin t)\in Q_{m_2}\\0,&\;\;\;\text{otherwise}\end{cases}.
	\end{align*}
	Denote 
	$\mathcal{E}(\alpha_{n_1},\beta_{n_2},\varepsilon)=\{(x,y,t):\;|\alpha_{n_1}(x+\cos t,y)\sin t+\beta_{n_2}(x,y+\sin t)\cos t|\leq \varepsilon\}.$ By the disjointness of the collection $\mathcal{Q}$ and support of $K_{{m}}$, the second term is dominated by,
	\begin{align*}
		&\iiint_{\R^3}\limits\sum_{\substack{{m}\in\mathcal{\widetilde{M}}:\\Q_{m_1}\in\mathcal{Q},\;Q_{m_2}\in\mathcal{Q}}}\chi_{Q_{m_1}}(x+\cos t,y)\chi_{Q_{m_2}}(x,y+\sin t)\\
		&\hspace{3.2cm}\times \chi_{({\operatorname*{supp}\zeta})\cap \mathcal{E}(\alpha_{n_1},\beta_{n_2},2\delta\lambda^{\gamma+\mu_1+\rho-1})}(x,y,t)\;dxdydt\\
		&\leq\mathcal{E}(\alpha_{n_1},\beta_{n_2},2\delta\lambda^{\gamma+\mu_1+\rho-1}),
	\end{align*}
	and the proof of \Cref{lemma:reductiontosublevel} concludes.
	\section{Estimates for Sublevel Set: Proof of \Cref{Sublevelsetestimate}}\label{sublevel:sec}
	To prove \Cref{Sublevelsetestimate} we will require certain preliminary estimates. The first lemma is used to refine a sublevel set to a new parameterized set of comparable measure to the original sublevel set. To be precise, we have the following.
	\begin{lemma}\label{refinelemma}
		Let $I=[\delta,2\delta]$, $E \subseteq [0,1]^2 \times I$ and $\Gamma:I\to\R^2$ be a compact curve such that $\Gamma(I)\subset [-c,c]^2$. Define 
		\begin{align*}
			E'&=\left\{{z} \in \mathbb{R}^{2}: |\{t\in I:({z}-\Gamma(t),t) \in E\}|\geq\frac{|E|}{8c^2}\right\},\\
			\text{and}\quad E_1&=\left\{({z},t)\in \R^2\times I \subset E: {z}\in E' \text{ and }({z}-\Gamma(t),t)\in E\right\}.
		\end{align*}
		Then $|E'|\geq\frac{|E|}{2\delta}$ and $|E_1|\geq\frac{|E|^2}{16c^2\delta}$.
	\end{lemma}
	\begin{proof}
		We can write 
		\[|E|=\int \chi_{E}({z}, t)\;dtd{z}.\]
		By a change of variable ${z}\to {z}-\Gamma(t)$, we get
		\begin{align*}
			|E| & =\int_{E'} \int_{I} \chi_{E}({z}-\Gamma(t), t)\;dtd{z}+\int_{\R^{2}\setminus E^{\prime}} \int_{I} \chi_{E}({z}-\Gamma(t), t)\;dtd{z} \\	&\leq\delta\left|E'\right|+\int_{[-c,c]^2}\frac{|E|}{8c^2}\;d{z}= \delta|E'|+\frac{|E|}{2}.
		\end{align*}
		The above equation implies $|E'|\geq\frac{|E|}{2\delta}$. Similarly, we can write
		\begin{align*}
			\left|E_{1}\right|& =\int_{\R^{2}} \int_{I} \chi_{E}({z},t) \chi_{E'}({z}) d t d{z} \\
			& \geq \int_{\R^{2}} \frac{|E|}{8c^2} \chi_{E'}({z}) d {z} \geq \frac{|E|^{2}}{16c^2 \delta}.
		\end{align*}
	\end{proof}
	We will also require the following quantitative version of inverse function theorem obtained in \cite[Lemma B.1]{CG}. The proof of the lemma is similar to that of the inverse function theorem in \cite[page 595]{Christ1985}.
	\begin{lemma}[\cite{CG}]\label{InverseFunctionTheorem}
		Let $\boldsymbol{T}=(T_1,T_2,T_3):\Omega\to\R^3$ be a $C^{(k)}(k \geqslant 2)$ function and for some $a \in \Omega$, we have
		\[|\operatorname{det}(\nabla \boldsymbol{T}(a))| \geqslant c>0,\;\text{and}\]
		\[\left|D^v T_i(x)\right| \leq C \text { for all } x \in \Omega,\;|\nu| \leqslant 2, \;i=1,2,3.\]
		If $r_0 \leq \sup \{r>0: B(a, r) \subset \Omega\}$ then $\boldsymbol{T}$ is a bijection from $B\left(a, r_1\right)$ to an open set $\boldsymbol{T}\left(B\left(a, r_1\right)\right)$ such that
		\[B\left(\boldsymbol{T}(a), r_2\right) \subset \boldsymbol{T}\left(B\left(a, r_1\right)\right) \subset B\left(\boldsymbol{T}(a), r_3\right),\;\text{where}\]
		\[r_1=\min \left\{\frac{c}{108C^3}, r_0\right\}, r_2=\frac{c}{24C^{2}} r_1 \text { and } r_3=\sqrt{3} C r_1.\]
		Moreover, the inverse mapping $\boldsymbol{T}^{-1}$ is also in $C^{(k)}$.
	\end{lemma}
	We now begin the proof of the sublevel set inequality in \Cref{Sublevelsetestimate}. Let $K'$ be image of $K$ under the map $(x,y,t)\mapsto (x-\sin t,y,t)$ and set ${z}=(x,y)$. It is enough to bound the measure of the set $\mathcal{E}(\epsilon)$ defined by
	\[\mathcal{E}(\epsilon)=\left\{({z}, t) \in K^{\prime}:|\alpha({z})-\cot(t) \beta({z}+\Gamma(t))| \leq \epsilon\right\},\]
	where $\Gamma(t)=\left(-\cos t, \sin t\right)$. We can assume that $K^{\prime} \subseteq[0,1]^{2} \times I$ which can be achieved by covering $K'$ by finitely many rectangles and translating each rectangle to $[0,1]^{2} \times I$. Note that we have $|\mathcal{E}(\epsilon)|<\delta$. We may also assume $|\mathcal{E}(\epsilon)|\neq0$, otherwise the theorem is trivial.
	
	As in \cite{CDR}, we employ the method of refinements to reduce the estimate of the measure of the set $\mathcal{E}(\epsilon)$ to estimating the measure of a set $\mathcal{A}\subset I^3$. We claim that there exist ${\bar{z}}\in[0,1]^2$ and a measurable set $\mathcal{A}\subset I^3$ such that $|\mathcal{E}(\epsilon)|\lesssim\delta^{\frac{4}{7}}|\mathcal{A}|^{\frac{1}{7}}$ and for every $(t_1,t_2,t_3)\in\mathcal{A}$, we have
	\begin{equation}\label{refine1}
		\left\{\begin{array}{l}
			\left({\bar{z}},t_{1}\right) \in \mathcal{E}(\epsilon), \\
			\left({\bar{z}}+ \Gamma(t_1)-\Gamma(t_2), t_{2}\right) \in \mathcal{E}(\epsilon), \\
			\left({\bar{z}}+\Gamma(t_1)-\Gamma(t_2), t_{3}\right) \in \mathcal{E}(\epsilon).
		\end{array}\right.
	\end{equation}
	The above conditions are equivalent to 
	\begin{equation}\label{refine2}
		\left\{\begin{array}{l}
			\left|\alpha({\bar{z}})-\cot\left(t_{1}\right) \beta\left({\bar{z}}+\Gamma(t_1)\right)\right| \leq \varepsilon, \\
			|\alpha\left({\bar{z}}+\Gamma(t_1)-\Gamma(t_2)\right)-\cot\left(t_{2}\right) \beta\left({\bar{z}}+\Gamma(t_1)\right)| \leq \varepsilon, \\
			\left|\alpha\left({\bar{z}}+\Gamma(t_1)-\Gamma(t_2)\right)-\cot\left(t_{3}\right) \beta\left({\bar{z}}+\Gamma(t_1)-\Gamma(t_2)+\Gamma(t_3)\right)\right| \leq \varepsilon.
		\end{array}\right.
	\end{equation}
	
	To obtain the above claim, we define
	\[\mathcal{E}_{0}^{\prime}=\left\{{z} \in \mathbb{R}^{2}: |\{t\in I:({z},t) \in \mathcal{E}(\epsilon)\}|\geq\frac{|\mathcal{E}(\epsilon)|}{2}\right\}\subset [0,1]^2.\]
	and
	\[\mathcal{E}_{1}=\left\{({z},t) \in \mathcal{E}(\epsilon):\; {z} \in \mathcal{E}_{0}^{\prime}\right\}.\]
	Then, using \Cref{refinelemma} for $E=\mathcal{E}$ with $\Gamma(t)=(0,0)$, we get that 
	\[\left|\mathcal{E}_{0}^{\prime}\right|\geq\frac{|\mathcal{E}(\epsilon)|}{2 \delta}\text{ and }\left|\mathcal{E}_{1}\right|\geq\frac{|\mathcal{E}(\epsilon)|^{2}}{4 \delta}.\]
	Next, we further refine the set $\mathcal{E}$.	Let
	\[\mathcal{E}_{1}'=\left\{{z} \in \mathbb{R}^{2}:\left|\left\{t \in I:({z}-\Gamma(t), t) \in \mathcal{E}_{1}\right\}\right| \geq\frac{\left|\mathcal{E}_{1}\right|}{8}\right\},\]
	and
	\[\mathcal{E}_{2}=\left\{({z},t) \in \mathbb{R}^{2} \times I:\; {z} \in \mathcal{E}_{1}^{\prime}\text{ and }({z}-\Gamma(t), t) \in \mathcal{E}_{1}\right\}.\]
	Applying \Cref{refinelemma} for $E=\mathcal{E}_1$ with $\Gamma(t)=(-\cos t,\sin t)$ and $\Gamma(I)\subset[-1,1]^2$, we get that
	\[\left|\mathcal{E}_{1}'\right|\geq\frac{|\mathcal{E}(\epsilon)|^2}{8 \delta^2}\text{ and }\left|\mathcal{E}_{2}\right|\geq\frac{|\mathcal{E}(\epsilon)|^{4}}{256\delta^{3}}.\]
	Lastly, we define
	\[\mathcal{E}_{2}'=\left\{{z} \in \mathbb{R}^{2}:\left|\left\{t\in I:({z}+\Gamma(t), t) \in \mathcal{E}_{2}\right\}\right| \geq \frac{\left|\mathcal{E}_{2}\right|}{8}\right\}.\]
	Then, from first part of \Cref{refinelemma}, we have $|\mathcal{E}_{2}'|\geq\frac{|\mathcal{E}(\epsilon)|^4}{512\delta^4}$.
	
	Now, fix arbitrary ${\bar{z}}\in\mathcal{E}_{2}'$ and denote
	\begin{align*}
		E&=\left\{t \in I:({\bar{z}}+\Gamma(t), t) \in \mathcal{E}_{2}\right\} \\
		E_{t_{1}}&=\left\{t \in I:\left({\bar{z}}+\Gamma\left(t_{1}\right)-\Gamma(t), t\right) \in \mathcal{E}_{1}\right\} \quad\text{for each }t_{1} \in E \\
		E_{t_{1},t_{2}}&=\{t \in I:({\bar{z}}+\Gamma(t_1)-\Gamma(t_2), t)\in \mathcal{E}\} \quad\text{for each }t_{1} \in E,\; t_{2} \in E_{t_{1}}.
	\end{align*}
	Note that these sets are well-defined as $E$, $E_{t_1}$ and $E_{t_1,t_2}$ are non-empty sets. It is easy to see that $|E|\gtrsim\frac{|\mathcal{E}(\epsilon)|^4}{\delta^3}$.
	
	Next, for $t_1\in E$, we know that $({\bar{z}}+\Gamma(t_1),t_1)\in\mathcal{E}_2$, which implies ${\bar{z}}+\Gamma(t_1)\in\mathcal{E}_1'$. Hence $|E_{t_1}|\gtrsim\frac{|\mathcal{E}(\epsilon)|^2}{\delta}$.
	
	Lastly, for $t_1\in E$ and $t_2\in E_{t_1}$, we can see that $({\bar{z}}+\Gamma(t_1)-\Gamma(t_2),t_2)\in\mathcal{E}_1$. This implies that ${\bar{z}}+\Gamma(t_1)-\Gamma(t_2)\in\mathcal{E}_0'$. Thus, we obtain $|E_{t_1,t_2}|\gtrsim|\mathcal{E}(\epsilon)|$.
	
	We define
	\[\mathcal{A}=\left\{\left(t_1, t_{2}, t_{3}\right)\in I^3:\; t_1\in E, t_{2} \in E_{t_1}, t_{3}\in E_{t_1,t_2}\right\}.\]
	Hence $|\mathcal{E}(\epsilon)|\lesssim\delta^\frac{4}{7}|\mathcal{A}|^\frac{1}{7}$ and the conditions in \eqref{refine1} are satisfied from the definition of $\mathcal{E}_1, \mathcal{E}_2, E, E_{t_1}$ and $E_{t_1,t_2}$. This concludes the proof of the claim.
	
	We now begin the estimate of the measure of $\mathcal{A}$. By triangle inequality and \eqref{refine2}, the function $J:I^3\to\R$ defined by
	\begin{equation*}
		J(t_1,t_2,t_3)=\alpha({\bar z})\cot(t_1)\tan(t_2)\cot(t_3)-\beta({\bar z}+\Gamma(t_1)-\Gamma(t_2)+\Gamma(t_3)),
	\end{equation*}
	satisfies $|J(t)|\lesssim\delta\epsilon$. This implies that $\mathcal{A}\subseteq\Omega$, where the set $\Omega$ is given by,
	\[\Omega:=\{(t_1,t_2,t_3)\in I^3:\;|J(t_1,t_2,t_3)|\leq C\delta\epsilon\}.\]
	We write $\Omega=\Omega_1\cup\Omega_2\cup\Omega_3\cup\Omega_4\cup\Omega_5$, where
	\begin{align*}
		\Omega_1&:=\left\{(t_1,t_2,t_3)\in\Omega:|t_1-t_2|\geq\frac{\delta\epsilon^a}{2},\;|t_3-t_1|\geq\frac{\delta\epsilon^a}{2}\;\text{and}\;|t_2-t_3|\geq\delta\epsilon^b\right\},\\
		\Omega_2&:=\left\{(t_1,t_2,t_3)\in\Omega:|t_2-t_3|\geq\frac{\delta\epsilon^a}{2},\;|t_1-t_3|\geq\frac{\delta\epsilon^a}{2}\;\text{and}\;|t_1-t_2|\geq\delta\epsilon^b\right\},\\
		\Omega_3&:=\left\{(t_1,t_2,t_3)\in\Omega:|t_1-t_2|\geq\frac{\delta\epsilon^a}{2},\;|t_2-t_3|\geq\frac{\delta\epsilon^a}{2}\;\text{and}\;|t_3-t_1|\geq\delta\epsilon^b\right\},\\
		\Omega_4&:=\left\{(t_1,t_2,t_3)\in\Omega:|t_1-t_2|\leq\delta\epsilon^a,\;|t_3-t_1|\leq\delta\epsilon^a\;\text{and}\;|t_2-t_3|\leq\delta\epsilon^a\right\},\\
		\Omega_4&:=\left\{(t_1,t_2,t_3)\in\Omega:|t_1-t_2|\leq\delta\epsilon^b,\;\text{or}\;|t_3-t_1|\leq\delta\epsilon^b\;\text{or}\;|t_2-t_3|\leq\delta\epsilon^b\right\},\\
	\end{align*}
	where the constants $0<a<b$ is to be determined later. It is easy to see that
	\begin{equation}
		|\Omega_4|\lesssim \delta^3\epsilon^{2a},\quad\text{and}\quad |\Omega_5|\lesssim \delta^3\epsilon^b.
	\end{equation}
	It remains to estimate the measure of the sets $\Omega_j,\;j=1,2,3$. By symmetry, it is enough to measure $\Omega_1$.
	
	Let $\boldsymbol{T}:I^3\to\R^3$ be the transformation given by
	\[\boldsymbol{T}(t_1,t_2,t_3)=(u,v,w)=(\cos(t_1)-\cos(t_2)+\cos(t_3),\sin(t_1)-\sin(t_2)+\sin(t_3),t_1).\]
	Then we have the Jacobian matrix of $\boldsymbol{T}$ satisfying,
	\begin{equation*}
		\left|\frac{\partial\boldsymbol{T}(t_1,t_2,t_3)}{\partial(t_1,t_2,t_3)}\right|=|\sin(t_2)\cos(t_3)-\sin(t_3)\cos(t_2)|=|\sin(t_2-t_3)|\geq\delta\epsilon^b.
	\end{equation*}
	Let $c_1,c_2,c_3>0$ be absolute constants determined by \Cref{InverseFunctionTheorem} such that the following holds true. 
	\begin{enumerate}
		\item There exists a collection of balls $\mathcal{B}$ of radius $c_1\delta^2\epsilon^{2b}$ centered at $c_1\delta^2\epsilon^{2b}{n}$ for ${n}\in\N^3$, which intersect with the set $\Omega_1$. Note that, we have that $\#\mathcal{B}\lesssim c_1^{-3}\delta^{-3}\epsilon^{-6b}$.
		\item For each $B=B((\bar t_1,\bar t_2,\bar t_3),c_1\delta^2\epsilon^{2b})\in\mathcal{B}$, we have a concentric ball $B^*=B((\bar t_1,\bar t_2,\bar t_3),c_2\delta\epsilon^b)$ radius $c_2\delta\epsilon^b$ such that
		\begin{itemize}
			\item For $(t_1,t_2,t_3)\in B^*$ we have $|t_1-t_2|>\frac{\delta\epsilon^a}{4},\;|t_3-t_1|>\frac{\delta\epsilon^a}{4}\;,\text{and}\;|t_2-t_3|>\frac{\delta\epsilon^b}{2}$.
			\item $\boldsymbol{T}:B^*\to\R^3$ is bijection and there exists $c_3>0$ such that
			\[\boldsymbol{T}(B)\subset B(\boldsymbol{T}((\bar t_1,\bar t_2,\bar t_3)),c_3\delta^2\epsilon^{2b})\subset\boldsymbol{T}(B^*).\]
		\end{itemize}
	\end{enumerate}
	
	Hence we have,
	\begin{align}
		\nonumber|\Omega_1\cap B|&=\int_B\chi_{\Omega_1}(t_1,t_2,t_3)\;dt_1dt_2dt_3\\
		\nonumber&=\int_{\boldsymbol{T}(B)}\chi_{\boldsymbol{T}(\Omega_3)}(u,v,w)\left|\frac{\partial (u,v,w)}{\partial(t_1,t_2,t_3)}\right|^{-1}\;dudvdw\\
		&\lesssim(\delta^2\epsilon^{2b})^2(\delta\epsilon^b)^{-1}\sup\limits_{(u,v)\in\Sigma_B}|\{w\in I_B^{(u,v)}:\;|J(\boldsymbol{T}^{-1}(u,v,w))|\leq C\delta\epsilon\}|,\label{intemediate3}
	\end{align}
	where $\Sigma_B=\{(u,v)\in\R^2:\;(u,v,w)\in\boldsymbol{T}(B)$ for some $w\in\R\}$, and $I_B^{(u,v)}=\{w\in\R:\;(u,v,w)\in B(\boldsymbol{T}(\bar t_1,\bar t_2,\bar t_3),c_3\delta\epsilon^{2a})\}$.
	\begin{claim}\label{claim:sublevel}
		For $(\bar u,\bar v)\in\Sigma_B$, we have
		\begin{equation*}
			|\{w\in I_B^{(\bar u,\bar v)}:\;|J(\boldsymbol{T}^{-1}(\bar u,\bar v,w))|\leq C\delta\epsilon\}|\lesssim\delta^3\epsilon^{1-2a}.
		\end{equation*}
	\end{claim}
	\begin{proof}[Proof of \Cref{claim:sublevel}:]
		Let $G(w)=J(\boldsymbol{T}^{-1}(\bar u,\bar v,w))=\alpha({\bar z})\cot(t_1)\tan(t_2)\cot(t_3)-\beta({\bar z}+(\bar u,\bar v))$. We observe that the following holds by the definition of $\boldsymbol{T}$.
		\[\frac{\partial t_1}{\partial w}=1,\quad\frac{\partial t_2}{\partial w}=\frac{\sin(t_1-t_3)}{\sin(t_2-t_3)},\quad\text{and}\quad\frac{\partial t_3}{\partial w}=\frac{\sin(t_1-t_2)}{\sin(t_2-t_3)}.\]
		Thus, we have that
		\begin{align*}
			|G'(w)|=&\bigg|\alpha({\bar z})\Big(-\cosec^2(t_1)\tan(t_2)\cot (t_3)\frac{\partial t_1}{\partial w}+\cot(t_1)\sec^2(t_2)\cot(t_3)\frac{\partial t_2}{\partial w}\\
			&\hspace{1mm}-\cot(t_1)\tan(t_2)\cosec^2(t_3)\frac{\partial t_3}{\partial w}\Big)\bigg|\\
			=&\bigg|\frac{\alpha({\bar z})}{4\sin^2(t_1)\cos^2(t_2)\sin^2(t_3)}\widetilde{J}(t_1,t_2,t_3)\bigg|,
		\end{align*}
		where the function $\widetilde{J}:I^3\to\R$ is defined as
		\begin{align*}
			\widetilde{J}(t_1,t_2,t_3)=&\frac{1}{\sin(t_2-t_3)}\big[\sin (2t_2)\sin(2t_3)\sin(t_2-t_3)+\sin (2t_1)\sin(2t_2)\sin(t_1-t_2)\\
			&\hspace{3cm}+\sin (2t_1)\sin(2t_3)\sin(t_3-t_1)\big].
		\end{align*}
		We have the following identity,
		\begin{align*}
			&\big[\sin (2t_2)\sin(2t_3)\sin(t_2-t_3)+\sin (2t_1)\sin(2t_2)\sin(t_1-t_2)\\
			&\hspace{2cm}+\sin (2t_1)\sin(2t_3)\sin(t_3-t_1)\big]\\
			=&-2 \sin \left(\frac{t_1-t_2}{2}\right) \sin \left(\frac{t_2-t_3}{2}\right) \sin \left(\frac{t_3-t_1}{2}\right)\\
			&\hspace{2cm}\times\big[\cos (t_1+t_2-2 t_3)+\cos (t_1-2 t_2+t_3)+\cos (-2 t_1+t_2+t_3)\\
			&\hspace{2.5cm}+\cos (2 t_1+t_2+t_3)+
			\cos (t_1+2 t_2+t_3)+\cos (t_1+t_2+2 t_3)\\
			&\hspace{2.5cm}+\cos (2 (t_1- t_2))+\cos (2 (t_2-t_3))+\cos (2 (t_3-t_1))\\
			&\hspace{2.5cm}+2 \cos (t_1-t_2)+2 \cos (t_2-t_3)+2 \cos (t_3-t_1)+1\big].
		\end{align*}
		Using the above identity it is clear that $|\widetilde{J}(t_1,t_2,t_3)|\gtrsim|t_1-t_2||t_3-t_1|\gtrsim\delta^2\epsilon^{2a}$. Using this and $|\alpha({\bar z})|\sim1$ (this is the only place where we require this hypothesis), we obtain that
		\[|G'(w)|\gtrsim\delta^{-2}\epsilon^{2a}.\]
		Hence, the above estimate and an application of mean value theorem produces the required claim.
	\end{proof}
	We employ the \Cref{claim:sublevel} along with \eqref{intemediate3} to obtain,
	\begin{align*}
		|\Omega_3|&\lesssim\sum\limits_{Q\in\mathcal{Q}}|\Omega_3\cap B|\\
		&\lesssim(\delta^{-3}\epsilon^{-6b})(\delta^2\epsilon^{2b})^2(\delta\epsilon^b)^{-1}(\delta^3\epsilon^{1-2a})\\
		&=\delta^3\epsilon^{1-2a-3b}.
	\end{align*}
	Therefore, we have $|\Omega|\lesssim\delta^3(\epsilon^{1-2a-3b}+\epsilon^{2a}+\epsilon^b)$. We now choose $a=\frac{1}{10}$ and $b=\frac{1}{5}$ to obtain $|\mathcal{A}|\leq|\Omega|\lesssim\delta^3\epsilon^{\frac{1}{5}}$. Thus, we obtain the bound $|\mathcal{E}(\epsilon)|\lesssim\delta\epsilon^{\frac{1}{35}}$ and the proof of \Cref{Sublevelsetestimate} concludes.
	
	\section{Necessary conditions for boundedness of $\mathfrak{A}_t$}\label{sec:examples}
	In this section, we discuss necessary conditions for the $L^{p_1}(\R^d)\times L^{p_2}(\R^d)\to L^p(\R^d)-$boundedness for the bilinear spherical average $\mathfrak{A}_t,\;t>0$. By scaling it is enough to consider the single average $\mathfrak{A}_1$.
	Let $\delta>0$ be a small number and $c>1$ be a fixed constant. We define $B(0,\delta)$ to be the ball with center $0\in\R^d$ and radius $\delta$, and $S^{\delta}_{a}(0)$ to be the $\delta-$neighborhood of sphere of radius $a$ and centered at origin, i.e. $\{x:||x|-a|<\delta\}$. We also define rectangles $R_1, R_2$ centered at origin with dimensions $[-5,5]^{2d-1}\times[-\delta,\delta]$ and $[-\delta,\delta]\times[-5,5]^{2d-1}$ respectively.
	
	For $\delta>0$, a small number, we will define functions $f, g$ with $\|f\|_{L^{p_1}},\|g\|_{L^{p_2}}\sim\delta^\alpha$ and a test set $E$ with $|E|\sim\delta^\beta$ satisfying 
	\[\mathfrak{A}_1(f_1,f_2)(x,y)\gtrsim \delta^\gamma, \;\;\;\text{for all}\;(x,y)\in E.\]
	This will imply that the necessary condition required for the operator $\mathfrak{A}_1$ to be bounded from $L^{p_1}(\R^d)\times L^{p_2}(\R^d)$ to $L^p(\R^d)$	is given by
	\[\alpha \leq \frac{\beta}{p}+\gamma.\]
	We now indicate the appropriate choice of the functions and the test sets along with the parameters $\alpha, \beta, \gamma$ in the figure below.
	\begin{figure}[H]
		\centering
		\resizebox{\textwidth}{!}{%
			\renewcommand{\arraystretch}{2}
			\begin{tabular}{|c|c|c|c|c|c|c|}
				\hline
				$f(x,y)$ & $g(x,y)$ & $E$ (test set) & $\alpha$ & $\beta$ & $\gamma$ & Necessary Condition \\
				\hline \hline
				\makecell{$\chi_{B(0,c\delta)}(x)$\\$\chi_{S^{c\delta}_{\frac{1}{\sqrt{2}}}(0)}(y)$} &
				\makecell{$\chi_{S^{c\delta}_{\frac{1}{\sqrt{2}}}(0)}(x)$\\$\chi_{B(0,c\delta)}(y)$} &
				$\chi_{S^{\delta}_{\frac{1}{\sqrt{2}}}(0)}\times\chi_{S^{\delta}_{\frac{1}{\sqrt{2}}}(0)}$ &
				$\frac{d+1}{p_1}+\frac{d+1}{p_2}$ & 2 & $2d-1$ &
				$\frac{d+1}{p_1}+\frac{d+1}{p_2}\leq 2d-1+\frac{2}{p}$ \\
				\hline
				\makecell{$\chi_{S^{c\delta}_{\frac{1}{\sqrt{2}}}(0)}(x)$\\$\chi_{B(0,c\delta)}(y)$} &
				\makecell{$\chi_{B(0,c\delta)}(x)$\\$\chi_{S^{c\delta}_{\frac{1}{\sqrt{2}}}(0)}(y)$} &
				$B(0,\delta)$ &
				$\frac{d+1}{p_1}+\frac{d+1}{p_2}$ & $2d$ & 1 &
				$\frac{d+1}{p_1}+\frac{d+1}{p_2}\leq 1+\frac{2d}{p}$ \\
				\hline
				$\delta^{\frac{-1}{p_1}}\chi_{R_1}(x,y)$ & $\chi_{[-5,5]^{2d}}(x,y)$ &
				$[-5,5]^{2d-1}\times[-\delta,\delta]$ & 0 & 1 & $-\frac{1}{p_1}$ &
				$\frac{1}{p_1}\leq \frac{1}{p}$ \\
				\hline
				$\chi_{[-5,5]^{2d}}(x,y)$ & $\delta^{\frac{-1}{p_2}}\chi_{R_2}(x,y)$ &
				$[-\delta,\delta]\times[-5,5]^{2d-1}$ & 0 & 1 & $-\frac{1}{p_2}$ &
				$\frac{1}{p_2}\leq \frac{1}{p}$ \\
				\hline
				$\chi_{B(0,\frac{1}{\delta})}$ & $\chi_{B(0,\frac{1}{\delta})}$ &
				$B(0,\frac{1}{\delta})$ & $\frac{-2d}{p_1}+\frac{-2d}{p_2}$ & $-2d$ & 0 &
				$\frac{1}{p}\leq \frac{1}{p_1}+\frac{1}{p_2}$ \\
				\hline
			\end{tabular}
		}
		\caption{Necessary conditions for $\mathfrak{A}_1$}
		\label{fig:necessary-conditions}
	\end{figure}

	\section*{Acknowledgement}
	We thank the anonymous referees for their valuable corrections and suggestions on the previous version of this article. Ankit Bhojak is supported by the Science and Engineering Research Board, Department of Science and Technology, Govt. of India, under the scheme National Post-Doctoral Fellowship, file no. PDF/2023/000708. Surjeet Singh Choudhary was supported by CSIR(NET), file no. 09/1020(0182)/2019- EMR-I for his Ph.D. fellowship and Indian Institute of Science Education and Research Mohali for Institute postdoctoral fellowship. Saurabh Shrivastava acknowledges the financial support from Science and Engineering Research Board, Department of Science and Technology, Govt. of India, under the scheme Core Research Grant, file no. CRG/2021/000230.
	\bibliography{bibliography}
\end{document}